\documentclass[leqno]{amsart} 
\usepackage{a4}
\usepackage{amsmath,amsthm}
\usepackage{amssymb,amscd,latexsym}
\usepackage{boxedminipage}
\usepackage{eepic}
\usepackage{epic}
\usepackage{wrapfig}
\usepackage{graphicx}

\providecommand{\bysame}{\makebox[3em]{\hrulefill}\thinspace}

\theoremstyle{plain}

\newtheorem*{def-theo}{Definition-Theorem}

\theoremstyle{definition}

\newtheorem*{definition}{Definition}

\theoremstyle{remark}
\newtheorem{remark}{Remark}

\newtheorem{thm}{Theorem}[section]

\newtheorem{lem}[thm]{Lemma}

\theoremstyle{definition}

\newtheorem{rem}[thm]{Remark}
\numberwithin{equation}{section}
\newcommand{\bp}{\begin{pmatrix}}
\newcommand{\ep}{\end{pmatrix}}
\newcommand{\bps}{\begin{smallmatrix}}
\newcommand{\eps}{\end{smallmatrix}}
\def\C{{\mathbb C}}

\def\R{{\mathbb R}}
\def\Z{{\mathbb Z}}

\def\E{{\mathcal E}}

\def\H{{\mathcal H}}
\def\I{{\mathcal I}}

\def\K{{\mathcal K}}

\def\O{{\mathcal O}}

\def\U{{\mathcal U}}

\def \0{{\bf 0}}
\def \1{{\bf 1}}

\def \Ker{\mathrm{Ker}}

\def \rank{\mathrm{rank}}

\def \mf#1#2#3#4{
\xymatrix{{#1}\  \ar@<0.4ex>[r]^{{#2}} & \ {#4}
\ar@<0.4ex>[l]^{{#3}}
}
}

\def \mfs#1#2#3#4{\!
\xymatrix@C=1.5em{{#1} \! \ar@<0.2ex>[r]^{{#2}} & \! {#4}
\ar@<0.2ex>[l]^{{#3}}
}
\!}

\def \mfl#1#2#3#4{
\xymatrix@C=2.6em{{#1}\  \ar@<0.4ex>[r]^{{#2}} &\  {#4}
\ar@<0.2ex>[l]^{{#3}}
}
}

\def \mfss#1#2#3#4{\!
\xymatrix@C=1.5em{{#1} \ar@<0.3ex>[r]^{{#2}} & {#4}
\ar@<0.3ex>[l]^{{#3}}
}
\!}

\begin{document}

{
\title[coherendce of direct images of the De Rham complex]{\!\!\! coherence of direct im\!a\!ge\!s of the D\!e\! Rham comple\!x\!\!\! 
}

\author{Kyoji Saito}

\bigskip
\address{Kavli IPMU (WPI), UTIAS, The University of Tokyo, Kashiwa, Chiba 277-8583, Japan}

\thanks{This work was supported by World Premier International Research Center Initiative (WPI), MEXT, Japan, and partially by JSPS Grant-in-Aid for Scientific Research (A) No. 25247004. }

\date{}
\maketitle

 {\renewcommand{\baselinestretch}{0.1}

\centerline{Dedicated to the memory of Egbert Brieskorn (7.7.1937-19.7.2013)}

\begin{abstract}\!
We\! show\! the\! coherence\! of\! the\! direct\! images\! of\! the\! De\! Rham complex  relative to a flat holomorphic map with suitable boundary conditions. 
For this 
purpose, a notion of bi-dg-algebra called the Koszul-De Rham algebra is developed.
\end{abstract}

\tableofcontents

\vspace{-1.0cm}
\section{Introduction} 
In the present paper, we prove the following theorem.

\smallskip
\noindent
{\bf Main Theorem.}
{\it 
Let $\Phi: Z\to S$ be a flat holomorphic map between complex manifolds.
\footnote{We assume that a manifold is connected, paracompact, Hausdorff and, hence, metrizable.}
Assume that there exists an open subset $Z'\subset Z$ with smooth boundary 
satisfying i) $Z'$ contains the critical set $C_{\Phi}$  of $\Phi$, ii) the closure $\bar{Z'}$ in $Z$ is proper over $S$, and iii) $Z'$ is a weak deformation retract of $Z$ along the fibers of $\Phi$ and iv) $\partial Z'$ is transversal to all fibers $\Phi^{-1}(t)$.  Then, the direct  images ${\R}^k\Phi_*(\Omega^\bullet_{Z/S}, d_{Z/S})$ 
of the relative De Rham complex $\Omega^\bullet_{Z/S}${\small$:=\Omega^\bullet_Z/\Phi^{*}(\Omega_S^1)\wedge \Omega^{\bullet-1}_Z$} on $Z$ over $S$ are  $\O_S$-coherent modules.
}

\smallskip
The main Theorem is well-known for a proper and/or projective morphism $\Phi$, 
since 1)  the $'E_1$-term $\mathrm{R}^q\Phi_*(\Omega_{Z/{S}}^p)$ of the spectral sequence defining the direct image, the so called {\it Hodge to De Rham spectral sequence} \eqref{spectral}, is already $\O_S$-coherent due to the proper mapping theorem of Grauert and/or Grothendieck,  and 2) the differentials on the spectral sequence  (induced from the relative De Rham differential $d_{Z/S}$)  are $\O_S$-homomorphisms  so that the limit of the spectral sequence is also $\O_S$-coherent (see \cite{Katz-Oda}\cite{Katz}). 

Therefore, our main interest is the study of the case when $\Phi$ is a non-proper morphism between open manifolds. 
We give a direct and down to the earth proof of the Main Theorem by introducing the notion of a {\it Koszul-De Rham algebra}, which seems to detect information of the singularities of the morphism $\Phi$ and to be of interest by itself (see {\bf Step} 3).
In such a non-proper mapping setting,  we also remark that the result has a close connection to a general theorem for  coherent $\mathcal{D}_Z$-modules which  are non-characteristic on the boundary by Houzel-Schapira \cite{Shapira1}, and its generalization to elliptic systems by Schapira-Schneiders (Theorem 4.2 in \cite{Schapira2}), since {\it the relative de-Rham system is an elliptic system}. 

If the range $S$ of $\Phi$  is one-dimensional, i.e.\ $\Phi$ is a function, and $Z$ is a suitably small neighborhood of an isolated critical point of $\Phi$, then the main Theorem was shown by Brieskorn \cite{Brieskorn1} and then by Greuel \cite{Greuel} (see Hamm \cite{Hamm} for what happens if non-isolated singularities are admitted). Namely, in the case of an isolated critical point, $\Phi$ is locally analytically equivalent to a polynomial map, and one proves the coherence by extending $\Phi$ to a projective morphism and then applying Grothendieck's coherence theorem for projective morphisms. The result was generalized by  the author in \cite{Saito2,Saito3} to the complete intersection case for higher dimensional base space $S$, where he did not use the above mentioned algebro geometric method in \cite{Brieskorn1} 
but used a complex analytic method developed by Forster and Knorr \cite{Forster-Knorr} who  gave  a new proof of the Grauert proper mapping theorem \cite{Grauert}. Recently, jointly with Changzheng Li and Si Li, the author studied in \cite{Li-Li-Saito} morphisms $\Phi$ which may no longer be defined locally in a neighborhood of an isolated critical point  but may have multiple critical points as in the Main Theorem. Then, $\Phi$ may no longer be equivalent to a polynomial map and the algebraic method in \cite{Brieskorn1} seems to be no longer applicable. However the analytic method in \cite{Saito2} can be generalized for this new setting, as will be presented in the present paper, where we study the De Rham cohomology group by the $\Check{\text{C}}$ech cohomology group with respect to an atlas \eqref{atlas2}
of relative charts due to Forster and Knorr. 

In the present  new setting, the morphism $\Phi$ may also no-longer neccesarily  have only isolated critical points but may have higher dimensional critical sets in the fibers of $\Phi$. For such semi-global settings, the vanishing cycles in the nearby fibers of $\Phi$ are no-longer purely  middle dimensional but mixed dimensional, and the De Rham cohomology groups are no-longer pure but mixed dimensional. Then, we need to solve  some topological problems. We also need to find a suitable Stein open covering of the fibration $\Phi$ in order to apply the Forster-Knorr result to the $\Check{\text{C}}$ech complex. This is achieved in the present paper by showing an existence of some enhanced structure on the atlas of relative charts $Z$ ({\it Lemma} \ref{based-lifting}).

\medskip
The  proof of the Main Theorem is divided into the following 4 steps. 

\medskip
\noindent
{\bf Step 1.}  We describe two (including Hodge to De Rham) spectral sequences, describing the direct images 
${\R}\Phi_*(\Omega^\bullet_{Z/S}, d_{Z/S})$ and see that the restriction from $Z$ to $Z'$  induces an isomorphism:  
$\R\Phi_*(\Omega_{Z/{S}}^\bullet,d_{Z/{S}}) \simeq \R\Phi_*(\Omega_{Z'/{S}}^\bullet,d_{Z'/{S}})$.  

\medskip
\noindent
{\bf Step 2.}
For any point $t\in S$, we find a Stein open neighborhood $S^*\subset S$ such that $Z'\cap \Phi^{-1}(S^*)$ is covered by atlases  of {\it relative charts} in the sense of Forster-Knorr, which satisfy  an additional condition, called complete intersection, and which form a family of atlases parametrized by the radius $r$ (${}^\exists r^*\le r \le 1$) of polydiscs.

\medskip
\noindent
{\bf Step 3.}\!\!
We introduce the Koszul-De Rham algebra $\K^{\bullet, \star}_{D(r)\times S^*/S^*,\bf f}$ on each relative 
chart $D(r)\times S^*$, as a sheaf of double dg-algebras over the dg-algebra $\Omega^\bullet_{D(r)\times S^*/S^*}$ of the relative De Rham complex, which gives an $\O_{D(r)\times S^*}$-free ``resolution" of the relative De Rham complex $(\Omega_{Z/{S}}^\bullet,d_{Z/{S}})$ up to the critical set $C_\Phi$, where the ``gap", i.e.\ the cohomology groups of  $\K^{\bullet, \star}_{D(r)\times S^*/S^*,\bf f}$ w.r.t.\ $\star$, is given by a sequence,  indexed by $s\in\Z_{\ge0}$, of complexes $(\mathcal{H}_\Phi^{\bullet,s},d_{DR})$ of coherent $\O_Z$-modules  supported in $C_\Phi$.

\medskip
\noindent
{\bf Step 4.}
The $\Check{\text{C}}$ech cohomology groups of the De Rham complex $(\Omega_{Z/{S}}^\bullet,d_{Z/{S}})$ and the  lifted $\Check{\rm{C}}$ech cohomology groups  of the  Koszul-De Rham algebra appear periodically in the first and the second terms of a long exact sequence of cohomology groups, where the second terms are coherent near $t\in S$ due to Forster-Knorr's result \cite{Forster-Knorr}.  The third terms of the sequence, described by the complexes $\mathcal{H}_\Phi^{\bullet,s}$ in Step 3, are also coherent on $S$, since $C_\Phi$ is proper over $S$. This shows that the first terms, i.e. the direct images of the De Rham complex, is also coherent near $t\in S$.


Since the coherence is a local property on ${S}$, this completes the proof.

\bigskip 
\noindent
\begin{remark} (i) A flat map $\Phi$ is an open map and defines a family of constant 
\[
n\ :=\ \dim_\C Z - \dim_\C {S}
\]
dimensional fibers.\ So, if $n\!=\!0$, the map $\Phi$ is proper finite and hence the Main Theorem is trivial. Therefore, in the present paper, we shall assume $n \!>\! 0$.\

\smallskip 
\noindent
 (ii) We introduce in the present note some tools which seem to be unknown in the literature: 
 
 \noindent
 a) The {\it atlases of some special intersection nature} ({\it Lemma} 3.2 and 3.3) in {\bf Step 2}, 
 
 \noindent
  b)\! The {\it sequence of chain complexes} $(\mathcal{H}^{\bullet,s}_\Phi,d_{DR})$ ($s\in\Z_{>0}$) of coherent $\O_Z$-modules supported in the critical set $C_\Phi$ of $\Phi$ in {\bf Step\! 3}.\!

   Both are essential for our purpose to give an analytic proof of the Main Theorem.
\end{remark}

\medskip
\noindent 
{\it Notation.} We  use cohomologies of three kinds: 1.\ De Rham complex,  2.\ derived functor of direct image $\Phi_*$, and 3.\ Koszul complex. According to them, when it is possible, we distinguish their indices by the following choices: 1.\ ``$\bullet$" or ``$p$" for $p\in \Z_{\ge0}$, 2.\ ``$*$" or ``$q$" for $q\in \Z_{\ge0}$, and 3.\ ``$\star$" or ``$s$" for $s\in \Z_{\ge0}$, respectively.

\medskip
\noindent
{\bf Acknowledgment}
The author expresses his gratitude to Changzheng Li, Si Li, Alexander Voronov, Mikhail Kapranov, Alexey Bondal, Tomoyuki Abe and Pierre Schapira for helpful discussions. The discussions with Changzheng Li and Si Li clarified the construction of the atlas in \S3, the discussions with Mikhail Kapranov and Alexander Voronov clarified the Koszul-de Rham algebra in \S4 and the discussions with Alexey Bondal and Tomoyuki Abe clarified the homological algebras in \S5. The author expresses gratitude to Dmytro Shklyarov, who after the present paper was submitted, informed author the paper \cite{B-P} \footnote
{In \cite{B-P} (Theorem 4.1),  some algebra similar to the (but differently graded) Koszul-De Rham algebra in the present paper was introduced in order to calculate the Hochshild and cyclic homology of a complete intersection affine variety (similar to the complete intersection variety $U$ in a Stein manifold $W$ in the present paper). It should be of interest to find a relation of the present work with the Hochshild and cyclic homology. However, since the description in \cite{B-P} misses the parameter space $S$ in the present paper, the relation seems not promptly apparent.} and pointed out some sign problem in the present paper.
He expresses also his gratitude to Scott Carnahan and Simeon Hellerman for their reading of the early version of the paper. 
Finally, the 
author thanks deeply an anonymous referee for carefully reviewing the manuscript. 

This work was supported by World Premier International Research Center Initiative (WPI), MEXT, Japan, and by JSPS Kakenhi Grant Number 25247004.

\medskip

\section{Step 1:  Hodge to De Rham spectral sequence}

\noindent
Throughout the present paper, we keep the setting and notation of the Main Theorem. 
Recall that the direct image is given by the hypercohomology $\R^\star\Phi_*(\Omega_{Z/{S}}^\bullet,d_{Z/{S}}) $ 
and is described by the limit of the following two spectral sequences: 
\begin{equation}
\label{spectral}
\begin{array}{ccl}
'E_2^{p,q}&:=& \mathrm{H}^p(\mathrm{R}^q\Phi_*(\Omega_{Z/{S}}^\bullet),d_{Z/{S}}) \\
''E_2^{q,p}&:=& \mathrm{R}^q\Phi_*(\mathrm{H}^p(\Omega_{Z/{S}}^\bullet,d_{Z/{S}})) .
\end{array}
\end{equation}
The $E_1$-term 
$'E_1^{p,q}= \mathrm{R}^q\Phi_*(\Omega_{Z/{S}}^p)$
of the first spectral sequence is sometimes called {\it Hodge to De Rham (or, Fr\"olicher) spectral sequence} for the De Rham cohomology relative to $\Phi$.

 Let us consider the second spectral sequence $''\!E_2^{q,p}$  \eqref{spectral}, which we shall denote also by $''\!E_2^{q,p}(Z/S)$ when we stress its dependence on the space $Z/S$.
  We first remark that $\mathrm{Supp}(\mathrm{H}^p(\Omega_{Z/{S}}^\bullet,d_{Z/{S}})) \subset C_\Phi$ for $p>0$ (here we recall that $C_\Phi$ is the critical set of $\Phi$ so that $\Phi\!\mid_{C_\Phi}$ is a proper morphism), since the Poincar\'e complex $(\Omega_{Z/{S}}^\bullet, d_{Z/{S}})$ relative to $\Phi$ is exact outside the critical set of $\Phi$.
 On the other hand,  we have $\mathrm{H}^0(\Omega_{Z/{S}}^\bullet,d_{Z/{S}})\simeq \Phi^{-1}\O_{{S}}$ (since $n>0$). That is, $\mathrm{H}^0(\Omega_{Z/{S}}^\bullet,d_{Z/{S}})$ is constant along fibers of  $\Phi$. Therefore, we observe (cf.\ \cite{Thom}):

\medskip
\noindent
{\bf Fact 1.}\  {\it 
Let $Z'$ be an open subset of $Z$ satisfying 
1. $C_F\subset Z'$ and 2. $Z'$ is a deformation retract of $Z$ along fibers of $\Phi$. 
Then, the inclusion map $Z'\to Z$ induces bijection $''\!E_2^{q,p}(Z/S) \simeq\     
''\!E_2^{q,p}(Z'/S)$\! 
and, hence, of the hypercohomology groups 
$\R^k \Phi_*\!(\Omega_{Z/{S}}^\bullet,\!d_{Z/{S}}) \simeq \R^k \Phi_*\!(\Omega_{Z'/{S}}^\bullet,\!d_{Z'/{S}})$ as $\O_{{S}}$-module.
}

\bigskip
\section{Step 2: Atlas of complete intersection relative charts}

We construct an atlas consisting of charts relative to the map $\Phi$ (called {\it relative charts} by Forster-Knorr \cite{Forster-Knorr}), which satisfy an additional condition called a complete intersection ({\it Lemma} 3.2).  The atlas shall be used in {\bf Step 4} to calculate the limit of the spectral sequence $'E_1^{p,q}$ by a generalization of $\check{\mathrm{C}}$ech cohomology. 
The construction of the atlas asks the existence of a certain covering of the manifold $Z$ of quite general nature ({\it Lemma} 3.3). Since the proof of the existence of such a covering is rather of technical nature and is independent of the other part of the paper,  hurrying readers are suggested to skip the present section and to go to \S4 after looking at definitions and results, and to come back to the proofs if necessary.

\medskip
\begin{definition} {\bf 1.}  (\cite{Forster-Knorr}) A {\it relative chart}  for a flat family $\Phi$ 
is a closed embedding 
\begin{equation}
j \ : \ U\ \longrightarrow \ D(r)\times {S}_U
\end{equation}
where $U$ is an open subset of 
$Z$ (which  may be empty), ${S}_U$ is an open subset of ${S}$ 
with $\Phi(U)\subset {S}_U$ and
$
D(r)
$ 
is a {\it polycylinder of the radius} $r$ \footnote
{A polycylinder of radius $r$ is by definition a domain of the form $\{(z_1,\cdots,z_m)\in\C^m \mid |z_i-a_i|<r\ (i=1,\cdots,m)\}$ where $(a_1,\cdots,a_m)\in \C^m$ is called the {\it center} of the polycylinder.
}
in some $\C^m$  ($m\in\Z_{\ge 0}$) 
such that the diagram 
\begin{equation} 
\label{relative.chart}
\begin{array}{ccccl}
\vspace{0.2cm}

U\ & &\!\! \overset{j}{-\!\!\! \longrightarrow} \!\!& &  D(r)\times {S}_U  \\
\vspace{0.2cm}

\Phi\mid_U \!\!\!&\searrow & & \swarrow &\!\!\! pr_{{S}_U} \\
&& {S}_U &&
\end{array}
\end{equation}
commutes. We sometimes  call the embedding $j$ a relative chart, for simplicity.
\end{definition}

\medskip
\begin{definition} {\bf 2.}
A relative chart is called a {\it complete intersection} if the $j$-image of $U$ is a complete intersection subvariety in $D(r)\times {S}_U$. That is, there exists a  sequence $f_1,\cdots, f_l$ of holomorphic functions on $D(r)\times {S}_U$, where $l$ is the codimension of $U$ in $D(r)\times S_U$:
\[
l:=m+\dim_{\C}S-\dim_{\C}Z=m- n
\]
\noindent
such that $j$ induces a  natural isomorphism $j^*: \O_{D(r)\times S_U}/(f_1,\cdots,f_l) \simeq \O_U$. 
\end{definition}

\medskip
\begin{lem}
{\it 
 Let $j_{k }:  U_{k} \rightarrow  D_{k}(r)\times {S}_k$ ($k\in K$) be a finite system of relative charts. 
 Then the fiber product 
 \begin{equation}
\label{productchart}
j_K: U_K \rightarrow  D_{K}(r) \times {S}_K
\end{equation} 
of the morphisms $j_{k}$ ($k\in K$) over $S_K:=\cap_{k\in K} S_{k}$, where we set 
 $D_K(r):=\prod_{k\in K} D_{k}(r)$,  is a relative chart.
 }
\end{lem}
\begin{proof}
The morphism $j_K$ is obviously a local embedding. We need to show that its image is closed. Suppose there is a sequence $z_i\in U_K$ {\small ($i=1,2,\cdots$)} such that the sequence $j_K(z_i)$ converges to a point in $D_K(r)\times S_K$. Then the projection sequence $j_k(z_i)$ also converges in $D_k\times S_k$, implying that the sequence $z_i$ converges in $U_k$ for all $k\in K$. Then $\underset{i}{\lim}\{x_i\}$ belongs to $\cap_{k\in K}U_k=:U_K$ 
(cf.\ \cite{Forster-Knorr} Cor.\ 3.2). 
\end{proof}

\begin{definition} {\bf 3.}
We shall call $j_K$ \eqref{productchart} the {\it intersection of relative charts} $j_k$ ($k\in K$).
\end{definition}

\smallskip
\begin{remark}
Let $\dim D_k=m_k$ and $l_k=m_k-n$ for $k\in K$. Then, $j_K(U_{K})$ has codimension equal to $l_K:=\sum_{k\in K}m_k -n=\sum_{k\in K} l_k+ (\#K-1)n$.
Even if all $j_k$ ($k\in K$) are complete intersections, their intersection $j_K$  may not necessarily be a complete intersection. Therefore, the following lemma is non-trivial.
\end{remark}

\bigskip
\begin{lem}
\label{complete-intersection-atlas}
{\it 
Let $\Phi:Z\to S$ be any flat holomorphic map. Then there exists a function $r: Z\to \R_{>0}$  and  a relative chart $j_z:U_z(r)\to D_z(r)\times S_z$ for all $z\in Z$ and $0<r< r(z)$ such that 1) $j_z(z)$ is independent of $r$ and  2) $p_1\circ j_z: U_z(r)\to D_z(r)$ is a bijection, mapping $z$ to the center of the polycylinder of radius $r$. Furthermore,  any finite intersection of these relative charts is  complete intersection.
}
\end{lem}

 \begin{proof}  
 We first provide the following lemma of a quite general nature.

\bigskip
\begin{lem} 
{\it 
Any complex manifold $M$ of dimension $N$ admits an atlas (=\! a collection of open charts covering $M$) such that, for any point of $M$, the union of charts containing the point is holomorphically embeddable into an open set in $\C^N$. 
 \footnote{A parallel statement obtained by replacing the terminologies: complex manifold, holomorphically and $\C^N$ by $C^\infty$-manifold, diferentiably and $\R^N$, respectively, holds by the same proof.}
}
\end{lem}
\begin{proof}
By the assumption on manifolds (see Footnote 1), $M$ is metrizable, and let $d$ be a metric on $M$.  For  $p\in M$ and $r\in \R_{\ge0}$, let $B(p,r):=\{q\in M\mid d(p,q)<r\}$ be the ball neighborhood of a point $p$ of radius $r$.
We define a function on $p\in M$ by  

\smallskip
\noindent
$R(p):=\sup\{ r\in \R_{\ge0}\mid B(p,r) \text{ is holomorphically embeddable
 in a domain in } \C^N\}$.\footnote{Here, embeddings are not necessarily isometric.}

\medskip
\noindent
Actually, $R$ is a positive valued continuous function on $M$ except if it takes constant value $\infty$.  For any fixed real number $b$ with $0\! <\! b\! <\! 1/3$, we show that the atlas $\{ (B(p,R(p)b), \varphi_p\}_{p\in M}$, where $\varphi_p$ is a holomorphic embedding of $B(p,R(p)b)$ into $\C^N$, has the desired property.\footnote
{For our later application, we may assume further more that $\varphi_p$ is extendable to a holomorphic embedding of the ball $B(p, R(p)(1\!-\!(1\!-\!3b)/(1\!+\!b))$ into $\C^N$ since $3b\!<\!1$ and $1\!-\!(1\!-\!3b)/(1\!+\!b))\!<\!1$.
}
Proof: Suppose $p\in M$ belongs to the chart $B(q,R(q)b)$ centered at $q\in M$. That means $d(p,q)<R(q)b$ and then $B(p,R(q)(1-b))\subset B(q,R(q)(1-b+b'))$ where $b':=d(p,q)/R(q)<b$ so that $1-b+b'<1$. Hence, the ball $B(q,R(q)(1-b+b'))$ is embeddable in $\C^N$, and so is $B(p,R(q)(1-b))$. This implies $R(p)\ge R(q)(1-b)$. 
On the other hand, for any small $\varepsilon>0$,
$B(q,R(p)-d(p,q)-\varepsilon)\subset B(p, R(p)-\varepsilon)$ is embeddable in $\C^N$,  one gets $R(q)\ge \lim_{\varepsilon\downarrow 0} (R(p)-d(p,q)-\varepsilon)=R(p)-d(p,q)$ and, hence, $(1+b)R(q)> (1+b')R(q)=R(q)+d(p,q)\ge R(p)$. 
Note that the chart $B(q,R(q)b)$ is contained in the ball $B(p,R(q)2b)$ of radius  $R(q)2b=R(q)(1-b)-R(q)(1-3b)$. Recalling $1-3b>0$ and inequalities $R(q)(1-b)\le R(p), \ R(q)>R(p)/(1+b)$, the radius is less than $R(p)1-R(p)(1-3b)/(1+b))=R(p)(1-(1-3b)/(1+b))$ which is a constant ($<R(p)$) independent of the point $q$. That is, all charts containing $p$ are covered by the same ball $B(p, R(p)(1-(1-3b)/(1+b))$ which is embeddable in $\C^N$.
\end{proof}
 
 We return to the proof of {\it Lemma} \ref{complete-intersection-atlas}. 
Let $\{ (B(\underline{z},R(\underline{z})b), \varphi_{\underline z}\}_{\underline z\in Z}$ be the atlas of $Z$ described in {\it Lemma} 3.3 (with the additional assumption of Footnote 5).  
 For any point $\underline{z}\in  Z$,  let $S_{\underline{z}}$ be a local coordinate neighborhood of $\Phi(\underline{z})$ in ${S}$.

Then, one finds easily a positive real number $r(\underline{z})$ such that for any real $r$ with $0<r<r(\underline{z})$, the polycylinder $D(r)$ of radius $r$ centered at $\varphi_{\underline{z}}(\underline{z})$ is contained  in the domain $\varphi_{\underline{z}}(B(\underline{z},R(\underline{z})b))\subset \C^{N}$  and  $\Phi(\varphi_{\underline{z}}^{-1}(D(r)))\subset {S}_z$.  
Then, 
$$ \begin{array}{rclcl}
j_{\underline{z}} &:& U_{\underline{z}}(r):=\varphi_{\underline{z}}^{-1}(D(r)) & \longrightarrow  & \ D(r)\ \times \ S_{\underline{z}} \\
& &\ \  \underline{z'} &\longmapsto &(\varphi_{\underline{z}}(\underline{z'}),\Phi(\underline{z'}))
\end{array}
$$ 
gives a family (parametrized by $r$) of relative chart centered at $z$.
The codimension $l$ of the image ${j}_{\underline{z}}(U_z(r))$ in $D(r)\times S_{\underline{z}}$ is equal to $m-n=N-n=\dim_\C S$. Actually, the image is determined  by a system of equations:
$$
\{t_i-\Phi_i\circ \varphi_{\underline{z}}^{-1}=0\}_{i=1}^{\dim_\C{S}},
$$ 
where $(t_1,\cdots,t_{\dim_\C {S}})$ is a local coordinate system of $S_{\underline{z}}$ and $\Phi_i$ is the $i$th coordinate component of the morphism $\Phi$.
Thus $j_{\underline{z}}$ is complete intersection.

Let us show that, for any finite set $K=\{(\underline{z},r_{\underline{z}})$ of $\underline{z}\in Z$ and $0<r_{\underline{z}}<r(\underline{z})$  such that $U_K:=\cap_{(\underline{z},r_{\underline{z}}) \in K}U_{\underline{z}}(r_{\underline{z}})$ (and, hence, $S_K:=\cap_{(\underline{z},r_{\underline{z}})\in K} S_{\underline{z}}$) is non-empty, the intersection relative chart $j_K: U_K\to D_K(r)\times S_K$ is complete intersection. 

Recall that $j_K$ is given by the fiber product morphism: 
$$
j_K\ :\ 
\underline{z}' \in U_K \longmapsto ( (\varphi_{\underline{z}}(\underline{z}'))_{(\underline{z},r_{\underline{z}})\in K}, \Phi(\underline{z}'))\in \prod_{(\underline{z},r_{\underline{z}})\in K}D_{\underline{z}}(r_{\underline{z}}) \times S_K,
$$ 
where the codimension of  $j_K(U_K)$ is  equal to $l_K= \#K\cdot \dim_\C S+ (\#K-1)n$.

In case of $U_K\not=\emptyset$, the existence of a point $\underline{z}_0\in U_K$ implies the inclusion:  
$$
\underset{(\underline{z},r_{\underline{z}})\in K}{\cup} U_{\underline{z}}(r_{\underline{z}}) \ \subset \underset{(\underline{z},r_{\underline{z}})\in K}{\cup} B(\underline{z}, R(\underline{z})b) \ \subset \ B(\underline{z}_0, R(\underline{z}_0)(1-\varepsilon))
$$ 
for {\small $\varepsilon:=(1-3b)/(1+b)$} ({\it Lemma} \ref{complete-intersection-atlas}). Let $z^1,\cdots,z^N$ be the coordinates of $\C^N$ where  the ball
$B(\underline{z}_0, R(\underline{z}_0)(1-\varepsilon))$ is embedded by extending the domain of $\varphi_{\underline{z}_0}$. We also denote by $\varphi_{\underline{z}}^{-1}$ ($\underline{z}\in K$) the composition map: {\small $D_K\times S_K\to D_{\underline{z}} \to U_{\underline{z}} \subset Z$}.

Then,  the image $j_K(U_K)$ is determined by the following two type of equations:

\medskip
\noindent
1) System equations for identifying polycylinders $D_{\underline{z}}(r_{\underline{z}})$ ($\underline{z}\in K$) with each other. 
That is, for each fixed $j$ with $1\le j\le N$, all  $z_j\circ \varphi_{\underline{z}}^{-1} $ ($\underline{z}\in K$) are equal to each other. 
There are {\small $(\#K-1)N=(\#K-1)(n+\dim_\C S)$} number of equations:
 $$
  z^j\circ \varphi_{\underline{z}_0}^{-1} = z^j\circ \varphi_{\underline{z}_1}^{-1}=\cdots\ =z^j\circ \varphi_{\underline{z}_{k^*}}^{-1}  \qquad  1\le j\le N \ .
 $$

\noindent
2) System equations for the graph of $\Phi$ on each polycylinder $D_{\underline{z}}(r_{\underline{z}})$ ($\underline{z}\in K$).
That is, for each fixed $i$ with $1\le i\le \dim_\C S$, $t_i= \Phi_i\circ \varphi_{\underline{z}}^{-1}$ for all $\underline{z}\in K$.  There are $\#K\cdot\dim_\C S$ number of equations.  However, after the identifications in 1), we do not need all equations but only for one point $\underline{z}\in K$: $t_i = \Phi_i\circ \varphi_{\underline{z}}^{-1}$ ($1\le i\le \dim_\C S$), that is, the number of necessary equation is equal to $\dim_\C S$. 

\medskip
Thus the total number of necessary equations is {\small $(\#K-1)(n+\dim_\C S)+\dim_\C S=\#K\cdot \dim_\C S_K +(\#K-1)n=\dim_\C(D_K\times S_K)-\dim_\C U_K$}, showing that the image $j_K(U_K)$ is a complete intersection subvariety of $D_K\times S_K$. 
It is also clear that the Jacobian of this system of  defining equations has constant maximal rank. \footnote
{To be precise, one need to show that any point in $D_K\!\times\! S_K$ satisfying the relations 1) and 2) is in the image of $j_K$. But this can be shown by a routine work so that we omit it.
}

This completes the proof of {\it Lemma} \ref{complete-intersection-atlas}.
\end{proof}

Recall  the domain $Z'\subset Z$ in the Main Theorem in \S1 Introduction. We assume that $\partial Z'$ in $Z$ is smooth and transversal to all fibers $\Phi^{-1}(t)$ for all $t\in S$. 

 \medskip
\noindent
{\bf Fact 2.}\  {\it 
For any point $t$ of ${S}$, there  exist 
a Stein open neighborhood ${S}^*$, a finite number of relative charts over ${S}^*$
\begin{equation}
\label{atlas}
j_k \ : \ U_k \ \longrightarrow \ D_k(1)\times {S}^*, \quad 0\le k\le k^*
\end{equation}
and a real number $0<r^*<1$ with the properties: for all $r$ with $r^*\le r\le 1$, set 
\begin{equation}
\label{z(r)}
U_{k}(r) :=  j_k^{-1}(D_k(r)\times {S}^*) \text{\quad and \quad}
Z'(r) := \cup_{k=0}^{k^*} U_k(r). \qquad
\end{equation}
Then, we have the following.

\smallskip
\noindent
1. 
One has the inclusions: $Z_{{S}^*}:=\Phi^{-1}({S}^*) \supset Z'(r) \supset Z'_{{S}^*}:=\Phi^{-1}({S}^*)\cap Z'$. 

\noindent
2.  $Z'(r)$ is retractible  to $Z'_{{S}^*}$ along fibers of $\Phi$.

\noindent
3. For any  $K\subset \{0,\cdots,k^*\}$, the relative chart $j_K$ is a complete intersection.
 }
 
 \medskip
\noindent
{\bf Corollary.}
For $r$ with $r^*\le r\le 1$, we have $\O_{S^*}$-isomorphisms
\begin{equation}
\label{r-independent}
\R^k \Phi_*\!(\Omega_{Z_{S^*}/{S^*}}^\bullet,\!d_{Z_{S^*}/{S^*}}) \simeq \R^k \Phi_*\!(\Omega_{Z(r)/{S^*}}^\bullet,\!d_{Z(r)/{S^*}}) .
\end{equation}

 \begin{proof}
For each point $z\in \bar{Z'}\cap \Phi^{-1}(t)$, we consider a relative chart ${j}_z: U_z(r) \to D_z(r)\times {S}_z$ of {\it Lemma} \ref{complete-intersection-atlas}. We consider two cases.

Case 1.  $z\in Z'$:  Choose any real $r$ such that $0<r<r(z)$ and ${U}_z(r)\subset Z'$. 

 Case 2. $z\in \partial Z'$: Choose any real $r$  such that $0<r<r(z)$ and $U_z(r')$ (as a manifold with corners) is transversal to $\Phi^{-1}(t)$ for all real $r'$ with $0<r' \le r$.

Since $\bar{Z'}\cap \Phi^{-1}(t)$ is compact, we can find a finite number of relative charts $\tilde j_k: \tilde U_k\to D_k(r_k)\times \tilde S_k$  ($0\le k\le k^*$) centered  at  points $\underline{z}_0,\cdots, \underline{z}_{k^*}$ on $\bar{Z'}\cap \Phi^{-1}(t)$ so that the union $\cup_{k=0}^{k^*} \tilde U_k$ contains the compact closure $\bar{Z'}\cap \Phi^{-1}(t)$. Then,  we can find a Stein open neighborhood  ${S}^*$ of $t$ such that 1) its compact closure $\bar S^*$ is contained in $\cap_{k=0}^{k^*} {S}_k$, 2) $\bar{Z'}\cap \Phi^{-1}(\bar{S^*})$ is contained in  $\cup_{k=0}^{k^*} \tilde U_k(r)$, and 3) all fibers $\Phi^{-1}(t')$ for $t'\in \bar S^*$ and $U_k(r')$ ($0<r'\le r_k$) for the chart $j_k$ whose central point $z_k$ is on the boundary $\partial Z'$. By a suitable rescaling of the coordinate system of charts, we may assume that all radii $r_k$ ($0\le k\le k^*$) are equal to 1.  Then, due the compactness of  $\bar S^*$, there exists a real number $r^*$ with $0<r*<1$ such that $\bar {Z'}\cap \Phi^{-1}(\bar{S^*})$ is contained in  $\cup_{k=0}^{k^*} \tilde U_k(r')$ for all $r'$ with $r^*\le r'\le 1$. Then, we introduce the relative  chart 
\eqref{atlas} by setting $U_k:=U_k\cap j_k^{-1}(D_k(1)\times S^*)$ and define  $Z'(r)$ as in \eqref{z(r)}. Then, 1. is trivial by definition,  2. is  a routine work, for instance due to R. Thom \cite{Thom}, and 3. is true since the system of relative charts $\{\tilde j_k\}_{k=0}^{k^*}$ has already this property ({\it Lemma} \ref{complete-intersection-atlas}). 
To see \eqref{r-independent}, we recall the argument done in {\bf Fact 1.}
\end{proof}

Let us briefly describe how these relative charts shall be used in the sequel.

For any Stein open subset $S' \subset S^*$ and any real number $r$ with $r^*\le r\le 1$, we first consider the atlas (a collection of charts)
\begin{equation}
\label{atlas1}
\U(r,S'):=\{(U_k(r,S'):=j_k^{-1}(D_k(r)\times S'),\varphi_k)\}_{k=0}^{k^*}
\end{equation} 
of $Z'(r,S'):=\cup_{k=0}^{k^*}U_k(r,S')$.  Actually, this is a Stein open covering, since the intersection $U_K(r,S'):
=\cap_{k\in K}U_k(r,S')$ for any subset $K\subset \{0,\cdots,k^*\}$ is isomorphic to a closed submanifold of $D_K(r)\times S'$ and, hence, is Stein.  Therefore, the $'E_1$-term of the Hodge to De Rham spectral sequence $\mathrm{H}^q(Z'(r,S'),\Omega^p_{Z/S'})$ is given by the  $\Check{\text{C}}$ech complex $(\check{\mathrm{C}}^*(\U(r,S'),\Omega^p_{Z/S'}),\Check{\delta})$ with respect to the atlas $\U(r,S')$. 

The atlas $\U(r,S')$ is lifted to an atlas of relative charts:
\begin{equation}
\label{atlas2}
  \mathfrak{U} (r,S'):=\{ j_k|_{U_k(r,S')} : U_k(r,S')\to D_k(r)\times S'\}_{k=0}^{k^*}.
\end{equation} 
In \S4, we construct double dg-algebras  $\mathcal{K}^{\bullet,\star}_{D_K(1)\times S'/S',\bf{f}}$ on $D_K(1)\times S'$ (depending on a choice of bases $\bf f$ of the defining ideal of $U_K(r,S')$ in $D_{K}(r)\times S'$) and a natural epimorphism $\pi: \mathcal{K}^{\bullet,\star}_{D_K(1)\times S'/S',\bf{f}} \to \Omega^\bullet_{U_K(r,S')}$, where the kernel of $\pi$ is described by the complex $(\mathcal{H}_\Phi^{\bullet,s})_{s>0}$ of coherent sheaves, whoes support is contained in the critical set $C_\Phi$
(we use here the complete intersection property of the relative charts).  
Then, in \S5 we construct a ``lifting"
$\check{\mathrm{C}}^*(\widetilde{\mathfrak{U}}(r,S'),\mathcal{K}^{\bullet,\star}_{D(r)\times S'/S',{\bf f}})$  of the 
$\Check{\text{C}}$ech complex (here, we need once again  to ``lift" the atlas $\mathfrak{U}(r,S')$ to a based lifting atlas $\widetilde{\mathfrak{U}}(r,S')$ (see {\it Lemma} \ref{based-lifting})), 
whose cohomology groups induces a coherent module in a neighborhood of $t\in S^*$ due to the Forster-Knorr  Lemma (see {\it Lemma} \ref{Forster-Knorr}). Since $ \mathcal{K}^{\bullet,\star}_{D_K(1)\times S'/S',\bf{f}}$, $\Omega^\bullet_{U_K(r,S')}$ and $(\mathcal{H}_\Phi^{\bullet,s})_{s>0}$ form an exact triangle, we obtain also the coherence of the direct image of $\Omega^\bullet_{U_K(r,S')}$.

\medskip
\section{Step 3: Koszul-De Rham algebras}

We introduce the key concept of the present paper, called the {\it Koszul-De Rham algebra}, which is a double complex of locally free sheaves over a relative chart and gives a free resolution of the relative De Rham complex $\Omega^\bullet_{U/S}$ up to $C_\Phi$.  

More precisely, we slightly generalize the relative chart \eqref{relative.chart} $j:U\to D(r)\times S_U$  to \eqref{relativechart2} $j:U\to W$,\footnote
{The generalization is done mainly for notational simplification replacing $D(r)\times S_U$ by $W$. 
 In application in \S5, we shall use relative charts only in the form \eqref{relative.chart}.}
and 
the {\it Koszul-De Rham-algebra}, denoted by $(\K_{W/S,{\bf f}}^{\bullet,\star},d_{DR},\partial_\K)$,\footnote
{ The notation might have  better been $\K_{j,{\bf f}}^{\bullet,\star}$ than $\K_{W/S,{\bf f}}^{\bullet,\star}$.
}
 is a sheaf on $W$ of bi-graded $\Omega^\bullet_{W/S}$-algebras
 equipped  with 1) the double-complex structure: De Rham operator $d_{DR}$ and Koszul operator $\partial_\K$
 and 2)  a natural epimorphism: $(\K_{W/S,{\bf f}}^{\bullet,\star},d_{DR},\partial_\K)$ $\to (\Omega^\bullet_{U/S}, d_{U/S})$.
 If the chart \eqref{relativechart2} is a complete intersection as in {\bf Step 2}, then the morphism gives a bounded $\O_W$-free resolution of $\Omega^\bullet_{U/S}$ up to some  ``error terms" $(\mathcal{H}_\Phi^{\bullet,s})_{s>0}$.

\medskip
We first slightly generalize the concept of the relative chart (3.1), (3.2). 

\medskip
\begin{definition} {\bf 4.}
A {\it based relative chart} $(j, {\bf f})$  is a pair of a holomorphic closed embedding $j:U\to W$ of a complex variety $U$ into a Stein variety $W$ with a commutative diagram over a Stein variety $S$:
\begin{equation} 
\label{relativechart2}
\begin{array}{ccccl}
\vspace{0.2cm}
U\!\!& &\!\!\!\!\! \overset{j}{-\!\!\! \longrightarrow} \! \!\!\!\!& &\!\!  W  \\
\vspace{0.1cm}
\Phi_U\!\!\!\!\!& \searrow & & \swarrow &\!\!\!\! \Phi_W \\
&& {S} && 
\end{array}
\end{equation}
and a finite generator system ${\bf f}=\{ f_1,\cdots,f_{l}\} \subset \Gamma(W,\O_W)$ of the defining ideal $\I_U$ of the image subvariety $j(U)$ in $W$ (i.e.~ $\I_U:=\ker(j_*j^*|_{\O_W})=\sum_i \O_W f_i$).
\end{definition}

\bigskip
In this setting, for $p\in\Z_{\ge0}$, there is a natural epimorphism $\pi=j_*j^*|_{\Omega^p_{W/S}}$
\begin{equation}
\label{QI}
                 \Omega_{W/{S}}^p   \overset{\pi}{\longrightarrow } j_*(\Omega_{U/{S}}^p) \ (\simeq \Omega_{U/S}^p) \ \to \ 0, \footnote
{For a notational simplicity, we  shall sometimes confuse  the sheaf $\Omega_{U/S}^p$ on $U$ with its $j$-direct image $j_*(\Omega_{U/{S}}^p)$ on $W$. For instance, we shall write
$\Omega_{W/{S}}^p/ \Sigma_{i=1}^l (
   f_{i} \cdot    \Omega_{W/{S}}^p \! + \!
  df_{i} \wedge \Omega_{W/{S}}^{p-1} ) \simeq \Omega_{U/{S}}^p$.
}
\end{equation}
between the K\"ahler differentials,
whose kernel, depending only on $\I_U$, is given by
\vspace{-0.1cm}
$$
  \overset{l}{\underset{i=1}{\large
   \Sigma}}    f_{i} \cdot    \Omega_{W/{S}}^p  + 
 \overset{l}{\underset{i=1}{\Sigma}} 
  df_{i} \wedge \Omega_{W/{S}}^{p-1} .
$$

We want to construct $\O_{W}$-free resolution of this ideal generated by $f_{i}$ ($1\le i \le l$) and by $df_{i}$ ($1\le i \le l$). We answer this problem, up to the critical set $C_\Phi$, by introducing the {\it Koszul-De Rham-algebra} $(\K_{W/S,\bf f},d_{DR},\partial_\K)$.
\!\footnote
{Usually, Koszul resolution is defined for even elements $f_i$'s, but here we construct a resolution for odd elements $df_i$'s together. The interpretation to regard it as the Koszul resolution for the odd elements $df_i$'s and to introduce the variables $\eta_i$ was pointed out by M.\ Kapranov, to whom the author is grateful.
}

\bigskip
\begin{definition} 
The {\it Koszul-De Rham-algebra} associated with the based relative chart \eqref{relativechart2} is 
a sheaf of {\it bi-dg-algebra}s $\K_{W/S,{\bf f}}$ over the dg-algebra 
$ \Omega_{W/{S}}^\bullet$
on $W$  equipped with two (co-)boundary operators $\partial_\K,\ d_{DR}$ and with bi-degrees, describe below.
\end{definition}

Consider a  sheaf on $W$ of  {\it graded commutative algebra}s 
over the dg algebra 
$\Omega_{W/{S}}^\bullet$
 \begin{equation}
 \label{doubledg}
 \mathcal{K}_{W/S,{\bf f}}:= \Omega_{W/{S}}^\bullet \langle \xi_1,\cdots,\xi_{l}\rangle[ \eta_1, \cdots,\eta_{l} ]\ /\ \I
\end{equation}
generated by indeterminates $\xi_1,\cdots,\xi_{l}, \eta_1, \cdots,\eta_{l}$, where $\xi_i$'s (resp.\ $\eta_i$'s) are considered as graded commutative odd (resp.\ even) variables in the following sense.

\  i) $\eta_i$'s and even degree differential forms on $W$ are commuting with all variables,   
 
 ii) $\xi_i$'s  and odd degree differentials forms on $W$ are anti-commuting with each other, and $\I$ is the both sided ideal generated by 
\begin{equation}
\label{even.odd}
\begin{array}{cl}
\quad \xi_i \xi_j+\xi_j \xi_i=0 \quad \text{and}\quad   \xi_i  \omega+\omega \xi_i=0 \quad \ \text{for }\ 1\le i,j\le l \ \text{and}\ \omega\in \Omega^1_W.\!\!\!\!\!\!
  \end{array}
\end{equation}

\medskip
We equip the algebra $\K_{W/S,{\bf f}}$ with the following 3 structures.

\medskip
1. {\bf the Koszul structure}:  We define Koszul boundary operator $\partial_{\K}$ on $\K_{W/S,{\bf f}}$  as the $\Omega_{W/{S}}^\bullet$-endomorphism of the algebra  defined by the relations
$$ 
\partial_{\K} \xi_i =f_{i}, \ \ \partial_{\K} \eta_i =-df_i  \ \ \text{ and }\  \ \partial_{\K}1=0.
$$
They automatically satisfy the relation: $\partial_{\K}^2=0$.

\smallskip
\noindent
{\it Proof.} The endomorphism $\partial_{\K}$ is well defined on the free algebra generated by $\xi_i$'s and $\eta_i$'s.  
Then, one checks that the endomorphism preserves the ideal $\I$ generated by relations \eqref{even.odd} (since $\partial_\K(\xi_i \xi_j+\xi_j \xi_i)=f_i\xi_j-\xi_if_j+f_j\xi_i-\xi_jf_i=0$ and  $\partial_\K(  \xi_i  \omega+ \omega \xi_i)=f_i\omega-\omega f_i=0 $), and, hence, induces the action $\partial_\K$ on the quotient $\K_{W/S,{\bf f}}$.  The relation $\partial_\K^2=0$ follows immediately from the facts $\partial_\K^2\xi_i=\partial_\K f_i=0$ and $\partial_\K^2\eta_i=-\partial_\K df_i=0$.
\quad { $\Box$}

\medskip
2.\! {\bf De Rham structure}: We regard $\K_{W/S,{\bf f}}$ as De Rham complex of the Grassmann algebra $\O_{W}\!\langle\xi_1,\!\cdots\!,\xi_{l}\rangle$  where $\xi_i$'s satisfy the first half of the Grassmann relations \eqref{even.odd}. 
Then the De Rham differential operator,  denoted by $d_{DR}$, acting on $\K_{W/S,{\bf f}}$ is given as an extension of the classical De Rham operator $d_{W/S}$ on $\Omega^\bullet_{W/S}$ by setting
$$
 \begin{array}{lll}
  d_{DR}= d_{W/S} + \sum_{j=1}^{l} \eta_j \partial_{\xi_j},
\end{array}
$$ 
where $\partial_{\xi_j}$ is the derivation of the Grassmann algebra with respect to the variable $\xi_j$. 
One, first, defines this operator as an endomorphism of the free algebra before dividing by the ideal $\I$. Then one check directly that the endomorphism preserves the ideal $\I$ (since $d_{DR}(\xi_i \xi_j+\xi_j \xi_i)=\eta_i\xi_j-\xi_i\eta_j+\eta_j\xi_i-\xi_j\eta_i=0$ and $d_{DR}(\xi_i  \omega+\omega \xi_i)=\eta_i\omega-\omega\eta_i=0$) so that it induces the required one  acting on $\K_{W/S,{\bf f}}$.
The second term of $d_{DR}$ switches odd variables $\xi_i$ to even variables $\eta_i$. We see easily the property $(d_{DR})^2=0$ follows from 
$$
 d_{DR}(\xi_j)=\eta_j, \ d_{DR}(\eta_j)=0 \ \text{ and } \ d_{DR}^2(\O_{S} )=0.  
$$

\noindent
De Rham differential and Koszul differentials are anti-commuting  with each other
$$
\partial_\K d_{DR}+d_{DR}\partial_\K=0 
$$
(since $(\partial_\K d_{DR}+d_{DR}\partial_\K)\xi_i= \partial_\K\eta_i+d_{DR}f_i=-df_i+df_i=0$ and $(\partial_\K d_{DR}+d_{DR}\partial_\K)\eta_i=0+d_{W/S}(df_i)=0$) 
so that the pair $(d_{DR}, \partial_\K)$ form a double complex structure on $\K_{W/S,{\bf f}}$.\footnote{
That this construction of De Rham structure is the universal construction of dg-structure on $\K_{W/S,{\bf f}}$ extending that  on $\Omega_{W\!/\!{S}}^\bullet$ was pointed out by A.\!Voronov, to whom  the author is grateful.}

\medskip
	3. {\bf Bi-degree decomposition}:  We give an $\O_W$-direct sum decomposition 
$$
\K_{W/S,{\bf f}}:=\K_{W/S,{\bf f}}^{\bullet,\star}
:=\oplus_{p\in\Z}\oplus_{s\in\Z} \K_{W/S,{\bf f}}^{p,s}
\vspace{-0.2cm}
$$
such that

i)  \quad  $\K_{W/S,{\bf f}}^{p,0}=\Omega_{W/S}^p \ \ (p\in\Z)$ \quad and \quad $\K_{W/S,{\bf f}}^{p,s}=0$ for either $p<0$ or $s<0$.

ii) \ \  $\partial_\K: \mathcal{K}^{p,s}_{W/S} \to \mathcal{K}^{p,s-1}_{W/S} \ \ \text{and} \  \  d_{DR}:\mathcal{K}^{p,s}_{W/S} \to \mathcal{K}^{p+1,s}_{W/S}$ \ \ ($p,s\in\Z$).

\medskip
\noindent
In order to achieve this, we introduce De Rham degree and Koszul degree on $\K_{W/S,{\bf f}}$.  
Namely, for a monomial 
of the form $\omega\Xi\E$ where $\omega\in \Omega_{W/{S}}^p$ ($p\in\Z_{\ge0}$) and $\Xi$, $\E$ are monomials in $\xi_1,\cdots,\xi_l$ and $\eta_1,\cdots,\eta_l$, respectively, we define degree maps
$$
\begin{array}{rll}
\deg_{DR}(\omega\Xi\E) \!\! & :=\ \text{the total degree as a differential form} \! &=\ p+\deg(\E)\\
{\deg}_{\K}(\omega\Xi\E) \!\! & :=\ \text{the total degree of the monomial $\Xi\E$}\! &=\ \deg(\Xi)+\deg(\E)
\end{array},
$$
where one should note that $\xi_j$'s are Grasmann variables and $\deg(\Xi)$s are bounded by $l$, but $\eta_i$'s are even variables and $\deg(\E)$s are un-bounded. If some monomials have the same degree with respect to $\deg_{DR}$ and/or ${\deg}_{\K}$, then we also call the sum of them homogeneous of the same degree with respect to $\deg_{DR}$ and/or ${\deg}_{\K}$.

The degree maps are additive with respect to the product in $\Omega_{W/{S}}^\bullet \langle \xi_1,\cdots,\xi_{l}\rangle [ \eta_1, \cdots,\eta_{l}] $. 
Since the ideal $\I$ is generated by bi-homogeneous elements \eqref{even.odd}, the bi-degrees 
$\deg_{DR}$ and $\deg_{\K}$ are induced on the quotient algebra $\K_{W/S,{\bf f}}$ \eqref{doubledg}. 
So, we set 
$$
\K_{W/S,{\bf f}}^{p,s}:=\{\omega\in \K_{W/S,{\bf f}}\mid \omega \text{ is bi-homogeneous with } (\deg_{DR},\deg_{\K})=(p,s)\}.
$$
for $p,s\in \Z$, where $\K_{W/S,{\bf f}}^{p,s}=0$ if either $p$ or $s$ is negative. Each graded piece $\K_{W/S,{\bf f}}^{p,s}$ is an $\O_W$-coherent module (since it is isomorphic to a direct sum of some $\Omega_{W/S}^\bullet$).

\bigskip
This completes the definition of Koszul-De Rham algebra. 
For short, we sometimes call the combination of these 3 structures the {\it bi-dg-algebra} structure.
In the following sections A), B), C), D) and E), we describe some basic properties and applications of Koszul-De Rham algebras.

\medskip
\noindent
{\bf A) Functoriality.}  

\noindent
The {\it functoriality of {\small Koszul-De Rham} algebra}, formulated in the following Lemma, shall play a key role in {\small Step 4} when we construct the lifted {\small $\Check{\text{C}}$ech} coboundary operator $\Check{\delta}$ (see \eqref{sign1}). 

\medskip
\begin{lem} 
\label{functorial}
{\it A morphism ${\bf w} : (j':U'\to W', {\bf f'}) \to (j:U\to W,{\bf f})$ between two based charts over the  same  base space $S$ is a pair $(w, h)$ of  1) a holomorphic map $w: W' \to W$ over $S$, whose restriction $u:=w\mid_{U'}$ induces a holomorphic map $U'\to U$ so that we have the commutative diagram:
\vspace{-0.2cm}
\[
  \begin{array}{rcl}
  \vspace{0.1cm}
  U' \!\! &-\!\!-\!\! \overset{j'}{-}\!\!\! {\longrightarrow}  \!&\! W'\\ 
|   & \searrow \quad \swarrow &\  | \\
 \mid u \!\!\!\!\! &  S & \ \mid w \\
  \downarrow  & \nearrow \quad \nwarrow  &\ \!\downarrow
  \vspace{-0.15cm} \\
  U \!  &-\!\!-\!\!\overset{j}{-}\!\!\longrightarrow  &  W  \quad,
  \end{array}
  \]
and a matrix $\{h_i^k\}_{i=1,\cdots,l}^{k=1,\cdots,l'}$ with coefficients in $\Gamma(W',\O_{W'})$ such that $w^*(f_i)=\sum_{k=1}^{l'} h_i^k f_k'$.
Then, the correspondence
\begin{equation}
\label{functorial2}
{\bf w}^\diamond (\xi_i):=\sum_k h_i^k \xi_k'   \quad \text{and}\quad 
{\bf w}^\diamond(\eta_i):=\sum_k h_i^k \eta_k'  + \sum_k dh_i^k \xi_k'.
\end{equation}
induces a bi-dg-algebra morphism 
\begin{equation}
\label{functorial3}
{\bf w}^\diamond\ : \quad  w^{*}\K_{W/S,{\bf f}}\ \longrightarrow \ \K_{W'/S,\bf f'}
\end{equation}
over the dg-algebra morphism $w^*: w^{*}\Omega_{W/S}^\bullet\! \to\! \Omega_{W'\!/S}^\bullet$.  The morphism is functorial in the sense that for a composition $\bf{w_1\circ w_2}$  of morphisms, we have
$$
(\bf{w_1\circ w_2})^\diamond \quad = \quad {\bf w_2}^\diamond \circ {\bf w_1}^\diamond
$$
}
\end{lem}
\begin{proof} 
The correspondence \eqref{functorial} induces a morphism between free agleblas before dividing by the ideal $\I$, which preserves the parities of the variabls and matches with the degree counting by $\deg_{DR}$ and $\deg_\K$. We have the correspondences
$$\xi_i\xi_j+\xi_j\xi_i \quad \mapsto \quad
\sum_{k,k} h_i^kh_j^l\ \big(\xi_k'\xi_l'+\xi_l'\xi_k'\big)
\vspace{-0.3cm}
$$ 
and 
$$
\xi_i\omega+\omega\xi_i \quad \mapsto \quad \sum_k h_i^k\ \big(\xi_k' w^*(\omega) +w^*(\omega) \xi_k'\big)
$$
so that the defining ideal $\I$ is preserved and the bi-graded algebra homomorphism ${\bf w}^\diamond$ \eqref{functorial3} is well-defined.

\noindent
The commutativity of ${\bf w}^\diamond$ with $\partial_{\K}$:
$$
\begin{array}{rcl}
{\bf w}^\diamond \partial_\K (\xi_i) &=&{\bf w}^\diamond(f_i)=\sum_k h_i^k f_k'=\sum_k h_i^k \partial_\K \xi_k'=\partial_\K(\sum_k h_i^k \xi_k')=\partial_K {\bf w}^\diamond (\xi_i),
\end{array}
$$
$$
\begin{array}{rcl}
{\bf w}^\diamond \partial_\K (\eta_i) &=&{\bf w}^\diamond(df_i)
=\sum_k d(h_i^k f_k') \\
&=& \sum_k( h_i^k df_k'+dh_i^k  f_k') 
= \sum_k \partial_\K\sum_k(h_i^k \eta_k'+ dh_i^k \xi_k')=\partial_K {\bf w}^\diamond (\eta_i).
\end{array}
$$
The commutativity of ${\bf w}^\diamond$ with $d_{DR}$:
$$
\begin{array}{rcl}
{\bf w}^\diamond d_{DR}(\xi_i) \!&\!=\!&\!{\bf w}^\diamond( \eta_i)=\sum_k w_i^k \eta_k' \! +\! \sum_k dw_i^k \xi_k'=d_{DR}(\sum_k w_i^k \xi_k' )=d_{DR}({\bf w}^\diamond(\xi_i)),
\end{array}
$$
$$
\begin{array}{rcl}
{\bf w}^\diamond d_{DR}(\eta_i)=0,  \!\!\!&\!\!\!\!\!&\!\!\!
d_{DR}{\bf w}^\diamond (\eta_i) =d_{DR}(\sum_k \! ( w_i^k \eta_k' \! +\! dw_i^k \xi_k')
=\sum_k\! ( dw_i^k \eta_k' \! - \! dw_i^k \eta_k')=0.
\end{array}
$$
The functoriality of ${\bf w}^\diamond$: consider a composition 
$(w,h)=(w_1,h_1)\circ(w_2,h_2)$ of two morphisms, where $w_1: W_2\to W_1$ and $w_2:W_3\to W_2)$ and $w_1^*(f_{i,1})=\sum_kh_{i,1}^kf_{k,2}$, $w_2^*(f_{k,2})=\sum_l h_{k,2}^l f_{l,3}$. So, $w=w_1\circ w_2$ and $h_{i}^l=\sum_k w_2^*(h_{i,1}^k)h_{k,2}^l$. Then, obviously,
$$
\begin{array}{rcl}
{\bf w}^\diamond(\xi_{i,1}) \!&=&\! \sum_l\big((\sum_k w_2^*(h_{i,1}^k)h_{k,2}^l)\xi_{l,3}\big)=
\sum_k w_2^*(h_{i,1}^k\xi_{k,2})  = w_2^\diamond(w_1^\diamond(\xi_{i,1})).
\end{array}
$$
$$
\begin{array}{rcl}
{\bf w}^\diamond(\eta_{i,1}) & = & \sum_l 
\Big(
\big(\sum_k 
 w_2^*(h_{i,1}^k)h_{k,2}^l \big) \eta_{l,2}
 +  d\big(\sum_k w_2^*(h_{i,1}^k)h_{k,2}^l \big) \xi_{l,2}
\Big ) \\
&=& \sum_l 
\sum_k 
 w_2^*(h_{i,1}^k) \big(h_{k,2}^l \eta_{l,2}
 +  dh_{k,2}^l \xi_{l,2} \big) +
 \sum_kd( w_2^*(h_{i,1}^k)) h_{k,2}^l) \xi_{l,2}\big) 
\big ) \\
&=& w_2^\diamond(\sum_k h_{i,1}^k \eta_{k,2}  + \sum_k dh_{i,1}^k \xi_{k,2} )
=w_2^\diamond(w_1^\diamond(\eta_{i,1})).
\end{array}
$$
\end{proof}

\noindent
{\bf B) Comparison $\pi$ with the De Rham complex  $(\Omega_{U/S}^\bullet,d_{U/S})$.}

We compare the dg-algebras $(\Omega_{U/S}^\bullet, d_{U/S})$ and $(\K_{W/S,{\bf f}}^{\bullet,\star},d_{DR},\partial_\K)$.
We summarize and reformulate well-known facts in terms of $\deg_\K$, $\partial_\K$ and $d_{DR}$ as follows.  

\medskip
\noindent
\begin{lem} 
\label{pi-comparison}
 {\it The  morphism $\pi$  \eqref{QI} satisfies the following properties.}

\noindent
1.  {\it The morphism $\pi$ induces an exact sequence:  }
  \begin{equation}
  \label{0th-cohomology}
  \mathcal{K}_{W/S,\bf f}^{\bullet,1}\  \overset{\partial_\K}{\longrightarrow} \    \mathcal{K}_{W/S,\bf f}^{\bullet,0}  \ \overset{\pi}{\longrightarrow} \    \Omega_{U/{S}}^\bullet \longrightarrow 0.  
\end{equation}

\noindent
2. {\it The morphism $\pi$ commutes with De Rham differentials:}
 \begin{equation}
 \label{pi-deRham}
\begin{array} {rcl}
\vspace{0.1cm}
 \mathcal{K}_{W/S,\bf f}^{\bullet,0}  & \overset{\pi}{\longrightarrow} &    \Omega_{U/{S}}^\bullet \\
\vspace{0.1cm}
 d_{DR} \downarrow\quad\  &  &\quad  \downarrow d_{U/S} \\
 \mathcal{K}_{W/S,\bf f}^{\bullet +1,0}  & \overset{\pi}{\longrightarrow} &    \Omega_{U/{S}}^{\bullet+1}
 \end{array}
 \end{equation}
  
 \noindent
 3.  {\it If there is a morphism ${\bf w}: (j',{\bf f}')\to (j, \bf f)$ between two based relative charts, then the morphism $\pi$ gives natural transformation between the functors ${\bf w}^\diamond$ and $u^*$. }
 \begin{equation}
 \label{pi-natural}
\begin{array} {rcl}
\vspace{0.1cm}
 w^* \mathcal{K}_{W/S,\bf f}^{\bullet,0}  & \overset{\pi}{\longrightarrow} & \!  u^* \Omega_{U/{S}}^\bullet \\
\vspace{0.1cm}
 w^\diamond \downarrow\quad\  &  &\  \downarrow u^* \\
 \mathcal{K}_{W'/S,{\bf f}'}^{\bullet ,0}  & \overset{\pi'}{\longrightarrow} &    \Omega_{U'/{S}}^{\bullet}
 \end{array}
 \end{equation}
 \end{lem}
 \begin{proof}  We have only to check the following facts.

\smallskip
\noindent
1) The complex $(K^{\bullet,0}_{W/S,{\bf f}},d_{DR})$ coincides with the De Rham complex $(\Omega^\bullet_{W/S},d_{W/S})$.

\noindent
2) The image $\partial_\K(\K^{\bullet,1}_{W/S,\bf f})$ in $\K^{\bullet,0}_{W/S, \bf f}$  of the $\deg_\K=1$ part of the algebra is equal 

to the ideal generated by $f_1,\cdots,f_{l}$ and $df_1,\cdots,df_{l}$ in $(\Omega_{W/{S}}^\bullet , d_{W/{S}})$ so that 

as $\O_U$-module, they are isomorphic.

\noindent
3)  The relative De Rham differential $d_{U/S}$ on $U$ is coincides with the one induced 

from the relative De Rham differential $d_{W/S}$ on $W$.

\noindent
4) The morphism $w^\diamond$ on the  $\deg_\K=0$ part of Koszul-De Rham algebra coincides 

with the pull-back morphism $w^*$ of differential forms.
\end{proof}

The first and the second properties of {\it Lemma} \ref{pi-comparison} means that the morphism $\pi$  induces a quasi equivalence of the  Koszul-De Rham double complex with the De Rham complex,
and the third property 3. means the naturallity of $\pi$, which shall be used, in the next section, when we compare the $\Check{\text{C}}$ech-triple complex with coefficients in $\mathcal{K}_{W/S,\bf f}^{\bullet,\star}$ with the $\Check{\text{C}}$ech-double complex with coefficients in $\Omega_{U/{S}}^\bullet$.

\bigskip
\noindent
{\bf C)} {\bf Boundedness of the Koszul-De Rham algebra.}

We discuss a certain {\it boundedness} of the Koszul-De Rham complexes, which is crucially used in the study of triple complex in the next section 5. 

Since there are no relations mixing  $\xi$ and $\eta$,  by definition, the bi-degree  $(p,s)$ term of the Koszul-De Rham algebra has the following direct sum decomposition.
\begin{equation}
\label{p-s-term}
\K_{W/S,{\bf f}}^{p,s}=\overset{\min\{p,s\}}{\underset{a=0}{\oplus}} \!\!\!
\ \
\underset{\substack{
\Xi \text{ is a monomimal in} \\
\text{$\xi_i$'s  of }  
\deg (\Xi)=s-a.
}}{\oplus} 
\ \
\underset{\substack{
\E \text{ is a monomial in}\\
\text{$\eta_i$'s of } 
  \deg (\E)=a.
}}\oplus
\ \
 \Omega^{p-a}_{W/S}\ \Xi\ \E.
\end{equation}

\medskip
\begin{lem}
\label{bound1}
{\it The set $\{(p,s)\in \Z^2\mid \mathcal{K}_{W/S,{\bf f}}^{p,s}\not=0\}$ is contained in the strip }
\begin{equation}
\label{bound1}
\{(p,s) \in \Z^2 \mid   -l\le p-s\le \dim_\C W \}
\end{equation} 
\end{lem}
\begin{proof} Suppose that there exists a nontrivial element $\K^{p,s}_{W/S,\bf f}\ni \omega\Xi\E\not=0$. Then, 
$p-s= \deg(\omega)-\deg(\Xi)$ (recall the definition of bi-degrees), where  $0\le \deg(\omega)\le \dim_\C W$ and $0\le \deg(\Xi)\le l$.    This gives the bound in the formula.  
\end{proof}

\begin{rem} 
\label{bound2} {\it Lemma} implies that the total Koszul-De Rham complex $\K^{\tilde \bullet}_{W/S,{\bf f}}$ is bounded. 
However each term $\underset{p-s=\tilde\bullet}{\oplus} \K_{W/S,{\bf f}}^{p,s}$
of the total complex  is an infinite sum, since $\eta_i$'s are even variables and the multiplication of any high power of them are non-vanishing and increases simultaneously the degrees $p$ and $s$. Nevertheless, such simple repetition of same terms (in stable area) seems harmless as we shall see, in  the next Step 4, that by taking 
a lifting of $\Check{\text{C}}$ech cohomology groups with coefficients in Koszul-De Rham algebras,  they can be truncated in \eqref{truncation}. 
\end{rem}

\medskip
\noindent
{\bf D) $\partial_\K$-cohomology group of Koszul-De Rham algebra.}

 For each fixed $p\in \Z_{\ge 0}$,  we study the cohomology of the bounded complex: 
\begin{equation}
\label{resolution}
0 \to  \mathcal{K}_{W/S,\bf f}^{p,p+l}    \!\overset{\partial_\K}{\longrightarrow} \cdots  \!\overset{\partial_\K}{\longrightarrow}   \mathcal{K}_{W/S,\bf f}^{p,3}   \!\overset{\partial_\K}{\longrightarrow}   \mathcal{K}_{W/S,\bf f}^{p,2}  \!\overset{\partial_\K}{\longrightarrow}    \mathcal{K}_{W/S,\bf f}^{p,1}   \!\overset{\partial_\K}{\longrightarrow}     \mathcal{K}_{W/S,\bf f}^{p,0} \to  0 .
 \end{equation} 
Let us  first fix a notation: for $p,s\in \Z$, set 
\begin{equation}
\label{Hps}
\mathcal{H}^{p,s}_{W/S,{\bf f}}:=\Ker\big(\partial_\K:\mathcal{K}^{p,s}_{W/S,{\bf f}} \to \mathcal{K}^{p,s-1}_{W/S,{\bf f}}\big)\ \big/\ \partial_\K \big( \mathcal{K}^{p,s+1}_{W/S,{\bf f}}\big),
\end{equation}  
and call it {\it Koszul-cohomology}, or $\partial_\K${\it -cohomology}. 

We first  recall some functorial properties of them. 

\bigskip
\begin{lem}
\label{Koszul-cohomology}

i) {\it The De Rham operator $d_{DR}$ on the Koszul-De Rham algebra induces
$$
d_{DR}\ : \quad 
\mathcal{H}^{p,s}_{W/S,{\bf f}} \ \longrightarrow \ \mathcal{H}^{p+1,s}_{W/S,{\bf f}}
$$ 
such that $d_{DR}^2=0$, which we shall call the De Rham operator on $\partial_\K$-cohomology.}

\smallskip 
ii) {\it Let ${\bf w}=(w,h): (j',{\bf f}')\to (j,{\bf f})$ be a morphism between based relative charts, and set $u:=w|_U' :U'\to U$. Then, the morphism $\bf w^\diamond$ \eqref{functorial3} induces 
 a morphism
 \[
 u^\diamond \ :\quad \mathcal{H}^{p,s}_{W/S,{\bf f}} \ \longrightarrow \ \mathcal{H}^{p,s}_{W'/S',{\bf f}'}
 \]
 which commutes with the De Rham operators on $\partial_\K$-cohomologies, and has the functorial property:  $(u_1\circ u_2)^\diamond= u_2^\diamond\circ u_1^\diamond$.}
 \end{lem}
\begin{proof} All these facts are immediate consequences of the fact that $\K^{\bullet,\star}_{W/S,\bf f}$ is a double complex with respect to $\partial_K$ and $d_{DR}$ shown in \S4, and the fact that $\bf w^\diamond$ is a bi-dg-algebra homomorphism from $\K^{\bullet,\star}_{W/S,\bf f}$ to $\K^{\bullet,\star}_{W'/S,\bf f'}$ ({\it Lemma} \ref{functorial}).
\end{proof}

Note that the $\partial_K$-cohomology groups are $\O_W$-coherent  modules, since  the modules $\mathcal{K}^{p,s}_{W/S,{\bf f}}$ are $\O_W$-coherent and the $\partial_\K$ are $\O_W$-homomorphisms (\eqref{p-s-term}). We analyze the $\partial_\K$-cohomologies in that context. The first basic fact is that they are  defined on $U$.

\bigskip
\begin{lem} 
\label{coherent}
{\it The $\partial_\K$-cohomology $\mathcal{H}^{p,s}_{W/S,{\bf f}}$ ($p,s\in \Z$) is an $\O_U$-coherent module.}
\end{lem}
\begin{proof} 
Recall that the defining ideal of $U$ is given by $\I_U=\sum_i \O_W f_i$. Therefore, to be an $\O_U$-module, we have only to show that $f_i \mathcal{H}^{p,s}_{W/S,{\bf f}}=0$. Let $\omega \in \K^{p,s}_{W/S,{\bf f}}$ be a presentative of an element $[\omega]\in \mathcal{H}^{p,s}_{W/S,{\bf f}}$ such that $\partial_\K\omega=0$. Then we calculate that $\partial_\K(\xi_i \omega)=\partial_\K(\xi_i) \omega+\xi_i\partial_\K\omega= f_i\omega$. That is, the class $[f_i\omega]\in \mathcal{H}^{p,s}_{W/S,{\bf f}}$ is equal to zero. 
\end{proof}
We know already by \eqref{0th-cohomology} that the zero-th $\partial_\K$-cohomology is naturally given by 
\begin{equation}
\label{Hp0}
\mathcal{H}^{p,0}_{W/S,{\bf f}} \quad \simeq \quad \Omega_{U/S}^p \
\end{equation}
which is compatible with the De Rham operator action.

In order to analyze the support of $\mathcal{H}^{p,s}_{W/S,{\bf f}}$ more carefully, 
recall the direct sum expression of the Koszul-De Rham algebra \eqref{p-s-term}.
We observe that the Koszul boundary operator $\partial_\K$ splits into a sum $\partial+\tilde \partial$, where each  $\partial$ and $\tilde\partial$ is defined as  $\Omega_{W/S}^\bullet$-linear endomorphisms such that
$$
\partial \xi_i= f_i \ , \ \partial \eta_i= 0 \ , \ \partial 1= 0 \qquad \text{and} \qquad \tilde\partial \eta_i= df_i\ , \ \tilde\partial \xi_i =0\ , \ \tilde\partial 1 =0\ .
$$
We see immediately the relations $\partial, \tilde\partial: \K_{W/S,{\bf f}}^{p,s}\to \K_{W/S,{\bf f}}^{p,s-1}$ for all $p,s\in \Z$ and 
 $\partial^2=\tilde\partial^2=\partial\tilde\partial+\tilde\partial\partial=0$. That is, for each fixed $p\in\Z$, 
the subcomplex $(\K_{W/S,{\bf f}}^{p,\star},\partial_\K)$ can be regarded as the total complex of a double complex $(\K_{W/S,{\bf f}}^{p,\star},\partial, \tilde\partial)$.

More precisely, let us denote by $\Omega_{W/{S}}^a\xi^b\eta^c$ the space  spanned by those elements of the form $\omega \Xi\E$ with  $\omega\in \Omega_{W/{S}}^a$, and $\Xi$ and $\E$ are monomials of $\xi_j$'s and $\eta_j$'s of degree ${\deg}(\Xi)=b$ and $\deg(\E)=c$, respectively.
Then, we have $\partial: \Omega_{W/{S}}^a\xi^b\eta^c \to \Omega_{W/{S}}^a\xi^{b-1}\eta^c$ and $\tilde\partial: \Omega_{W/{S}}^a\xi^b\eta^c \to \Omega_{W/{S}}^{a+1}\xi^{b}\eta^{c-1}$. So, by putting
$\K_{W/S,\bf f}^{p,\{b,c\}}:=\Omega_{W/S}^{p-c}\xi^b\eta^c$ for $p,b,c\in \Z$, we get double complex 
$(\K_{W/S,\bf f}^{p,\{\star,\tilde\star\}},\partial,\tilde\partial)$ (where 
$\K_{W/S,\bf f}^{p,\{b,c\}}\not=0$ only when $0\le b\le l$ and $0\le c\le p$), and we have the identification of the total complex with the original Koszul complex: 
\begin{equation}
\label{Koszul-total}
(\oplus_{b+c=\star}\K_{W/S,\bf f}^{p,\{b,c\}}, \partial+\tilde\partial) \quad =\quad (\K_{W/S,{\bf f}}^{p,\star}, \partial_\K)
\end{equation}
for each fixed $p\in\Z$.
Explicitly, the double complex is given in the following Table.

\vspace{-0.1cm}

\[
\begin{array}{ccccccccccccc}
\vspace{0.1cm}
\Omega_{W/\!{S}}^p &\!\! \overset{\tilde\partial}{\leftarrow} \!\!&\Omega_{ W/\!{S}}^{p-1}\eta^1 &\!\!\! \overset{\tilde\partial}{\leftarrow} \!\!& \cdots &\!\! \overset{\tilde\partial}{\leftarrow} \!\!\! &\Omega_{ W/\!{S}}^1\eta^{p-1} &\!\!\! \overset{\tilde\partial}{\leftarrow} \!\!\!& \O_{ W}\eta^p&  \!\!\! \leftarrow  0 \\
\uparrow \partial && \uparrow\partial& & && \uparrow \partial && \uparrow\partial& \\
\vspace{0.2cm}
\Omega_{ W/\! {S}}^p\xi^1 &\!\! \overset{\tilde\partial}{\leftarrow} \!\!&\Omega_{ W/\! {S}}^{p-1}\xi^1\eta^1 &\!\!\! \overset{\tilde\partial}{\leftarrow} \!\!& \cdots &\!\!\overset{\tilde\partial}{\leftarrow} \!\!\! &\Omega_{W/\! {S}}^1\xi^1\eta^{p-1} &\!\!\! \overset{\tilde\partial}{\leftarrow} \!\!\!& \O_{W}\xi^1\eta^p&  \!\!\leftarrow  0 \\
\vspace{0.1cm}
\uparrow \partial && \uparrow\partial& &  && \uparrow \partial && \uparrow\partial&  \\
\vspace{0.1cm}
\cdots                 &&   \cdots               &&   \cdots  &&  \cdots                && \cdots    \\
\uparrow \partial && \uparrow\partial& &  && \uparrow \partial && \uparrow\partial&  \\
\vspace{0.2cm}
\! \Omega_{ W/\! {S}}^p\xi^{l\!-\!1}\!\! & \overset{\tilde\partial}{\leftarrow} &\! \Omega_{ W/\! {S}}^{p-1} \xi^{l\!-\!1} \eta^1\!\! &\!\! \overset{\tilde\partial}{\leftarrow} \!\!& \cdots &\! \overset{\tilde\partial}{\leftarrow} \! &\!\Omega_{ W/\! {S}}^1\xi^{l\!-\!1} \eta^{p\!-\!1}\!\! &\! \overset{\tilde\partial}{\leftarrow} \!&\! \O_{W}\xi^{l\!-\!1} \eta^p \!\!&  \!\!\leftarrow  0 \\
\uparrow \partial && \uparrow\partial& &  && \uparrow \partial && \uparrow\partial&  \\
\vspace{0.2cm}
\Omega_{ W/\! {S}}^p\xi^{l} &\!\! \overset{\tilde\partial}{\leftarrow} \!\!&\Omega_{ W/\! {S}}^{p-1}\xi^{l}\eta^1 &\!\!\! \overset{\tilde\partial}{\leftarrow} \!\!& \cdots &\!\! \overset{\tilde\partial}{\leftarrow} \!\!\! &\Omega_{ W/\! {S}}^1\xi^{l}\eta^{p-1} &\!\!\! \overset{\tilde\partial}{\leftarrow} \!\!\!& \O_{W}\xi^{l}\eta^p&  \!\!\leftarrow  0 \\
 \uparrow  && \uparrow& &  && \uparrow  && \uparrow&  \\
  0             &&    0          &&            &&    0          &&      0
\end{array}
\vspace{-0.1cm}
\]
\centerline{    {\bf   Table:  \  Double complex $(\K_{W/S,\bf f}^{p,\{\star,\tilde\star\}}, \partial,\tilde\partial)$\qquad } }

\bigskip

\begin{definition} {\bf 5.}
A based relative chart $(j,{\bf f}) $ is called a {\it complete intersection} if its underlying relative chart \eqref{relativechart2} satisfies the following 1), 2) and 3).

1)  The varieties $U$, $W$ and $S$ are smooth.

2)  The map $\Phi_W: W\to S$ is submersive, in particular, $\Phi_W$ has no critical points.

3)  The $U$ is a complete intersection subvariety of $W$ and $f_1,\cdots,f_l$ is a minimal 

\quad system of equations for  $U$, i.e.~$f_1,\cdots,f_l$ form a regular sequence on $W$.
\end{definition}

\smallskip
From now on through the end of the present paper, we study only based relative charts which is  complete intersection.
For such relative chart, we say that the morphism $\Phi_U: U\to S$ is critical at a point in $U$ if $\Phi_U$ is not submersive at the point. That is, the variety of critical set is given by
\begin{equation}
\label{relative-critical}
C_{\Phi_U}:= \{ x\in U \mid \text{the } \rank \text{ of the Jacobian of } \Phi_U \text{ at } x \text{ is less than } \dim_\C S \}, \!\!\!
\end{equation}
whose defining ideal $\I_{C_{\Phi_U}}$ in $\O_U$ is the one generated by the minors  of size  $\dim_\C S$ of the Jacobian matrix of $\Phi_U$.  We now prove some basic properties of 
$\mathcal{H}^{p,s}_{W/S,{\bf f}}$ which we shall use in the next section seriously.

\bigskip
\begin{lem}
\label{f-free}
{\it Suppose that a based relative chart \eqref{relativechart2} is a complete intersection.
Then, we have 

\smallskip
\noindent
{\rm (1)}  the $\O_U$-module $\mathcal{H}^{p,s}_{W/S,{\bf f}}$ for $s,p\in \Z$ together with the action of De Rham operator $d_{DR}$ is independent  of the choice of basis {\bf f} but depends only on the morphism $\Phi_U$, 

\noindent
{\rm (2)}  the support of the module $\mathcal{H}^{p,s}_{W/S,{\bf f}}$ for $s>0$ is contained in the critical set $C_{\Phi_U}$. }
\end{lem}
\begin{proof}
Before we  start to prove this Lemma, we visit the double complex $\K_{W/S,{\bf f}}^{p,\{b,c\}}$ given in {\bf Table}  under the complete intersection assumption.

\medskip
By {\bf Definition} 1) and 2) of complete intersection, $\Omega_{W/S}^\bullet$ is an $\O_W$-locally free modules of finite rank.
The $i$th vertical direction (w.r.t. the coboundary operator $\partial$) of the diagram for $i=0,1,\cdots,p$ is the classical Koszul complex on the locally free module $\Omega_{W/S}^{p-i}$ for the regular sequence  $f_1,\cdots,f_l$ (recall that $\xi_i$'s are odd variables), which is exact except at the zeroth stage,  and the cokernel module at the zeroth stage is an $\O_U$ locally free module isomorphic to $\Omega^{p-i}_{W/S}\eta^i/(f_1,\cdots,f_l)\Omega^{p-i}_{W/S}\eta^i=\big(\overset{p-i}{\wedge} \Omega_{W/S}^1\eta^i\big)\otimes_{\O_W}\O_U$. 
Between the modules, $\tilde \partial$ induces a cochain complex structure, denoted again by  $\tilde \partial$. In view of \eqref{Koszul-total}, this chain complex 

$$
{(**)} \qquad\qquad \qquad \qquad\qquad \qquad  \big(\big(\overset{p-\star}{\wedge} \Omega_{W/S}^1\big)\eta^\star\otimes_{\O_W}\O_U, \ \tilde\partial \ \big)  \qquad\qquad \qquad \qquad\qquad \qquad \qquad\qquad \qquad 
$$ 
 is quasi-isomorphic to the Koszul complex $(\K_{W/s,{\bf f}}^{p,\star},\partial_\K)$.
 Therefore, we show that the cohomology groups of (**) does not depend on the choice of the bases ${\bf f}$.

We provide, now, the following elementary but quite useful reduction lemma. 

\bigskip
\begin{lem}  
\label{reduction}
{\it Let $f_1$ be the first element of the basis ${\bf f}=\{f_1,\cdots,f_l\}$. Suppose $df_1$ is a part of $\O_W$-free basis of the module $\Omega_{W/S}^1$. Consider the hypersurface  $W':=\{f_1=0\}\subset  W$ and set ${\bf f}'=\{f_2,\cdots,f_{l}\}$. Then, we have 

\noindent
{\rm (1)} $(j':U\to W',{\bf f}')$ is also a complete intersection based relative chart, 

\noindent
{\rm (2)} the inclusion map $\iota: W'\subset W$ together with the correspondence $\eta_1\mapsto 0$ induces a morphism between the based relative charts and a quasi-isomorphism of chain complexes of $\O_U$-modules:
$$
\big(\big(\overset{p-\star}{\wedge} \Omega_{W/S}^1\big)\eta^\star\otimes_{\O_W}\O_U,\tilde\partial \big) \to \big(\big(\overset{p-\star}{\wedge} \Omega_{W'/S}^1\big)\eta'^\star\otimes_{\O_{W'}}\O_U,\tilde\partial \big).
$$ 
\noindent
{\rm (3)} The $\O_U$-isomorphism:  $\mathcal{H}^{p,s}_{W/S,{\bf f}} \simeq  \mathcal{H}^{p,s}_{W'/S,{\bf f}'}$ ($p,s\in\Z$) obtained from this quasi-isomorphism coincides with $(\iota|_U)^\diamond$ (recall {\it Lemma} {\rm \ref{Koszul-cohomology}  ii)}).  In particular, the isomorphism commutes with the De Rham operator action.}
\end{lem}
\begin{proof} (1)  The fact that $df_1$ is a part of $\O_W$-free basis $\Omega_{W/S}^1$ implies that $W'$ is a smooth variety and that the restriction $\Phi_U':=\Phi_U|_{W'}$ is still submersive. 

\noindent
(2) On the space $W$, the two chain complexes of sheaves 
$\big(\big(\overset{p-\star}{\wedge} \Omega_{W'/S}^1\big)\eta'^\star\otimes_{\O_{W'}}\O_U,\tilde\partial \big)$ and  
  $\big(\big(\overset{p-\star}{\wedge} \Omega_{W/S}^1/\O_W df_1\big)\eta'^\star \otimes_{\O_W}\O_U,\tilde\partial \big)$ are naturally isomorphic, since $\O_U\simeq \O_W/(f_1,\cdots,f_l)\simeq \O_{W'}/(f_2,\cdots,f_l)$. 
  Therefore, in order to show (2), it is sufficient to show the following general linear algebraic facts (cf.~\cite{Saito11}).
  
  \medskip
\noindent
{\it Proposition.}
{\it Let $M$ be a free module of finite rank over a noetherian commutative unitary ring $R$. Let $\wedge^* M$ be the Grassmann algebra of $M$ over $R$. For given elements $\omega_1,\cdots,\omega_k$ of $M$, consider the polynomial ring $\wedge^*M[\eta]$ of $k$ variables $\eta_1,\cdots,\eta_k$ equipped with a Koszul differential $\tilde\partial$ defined by setting $\tilde\partial(\eta_i)=\omega_i$ on it.  

\noindent
(a) Let $\mathfrak{a}$ be the ideal in $R$ generated by the coefficients of $\omega_1\wedge\cdots\wedge \omega_k$. Then, $i$-th cohomology group of $(\wedge^*M[\eta], \tilde\partial)$ vanishes for $i<\mathrm{depth}(\mathfrak{a})$.

\noindent
(b) If $\omega_1$ is a part of some $R$-free basis system of $M$, then the natural chain morphism $(\wedge^*M[\eta], \tilde\partial) \to (\wedge^*M/R\omega_1 [\eta'], \tilde\partial')$  (where $\eta'$ is the indeterminates $\eta_2,\cdots,\eta_k$ such that $\tilde  \partial'(\eta_i)=\omega_i$ ($i=2,\cdots,k$) and $\eta_1$ is mapped to 0) is quasi isomorphic.}

 \medskip
\noindent
(3) First, we note that there is a slight abuse of notation. Namely, we have needed to fix the coefficient matrix $h$ of the transformation $\iota^*(f_1)=0$ and $\iota^*(f_i)=f_i$ for $i=2,\cdots,l$ in order that $(\iota,h)^\diamond: \K_{W/S,{\bf f}}^{p,s} \to \K_{W'/S,{\bf f}'}^{p,s}$ is defined (recall {\it Lemma} \ref{functorial}).  Once $(\iota,h)^\diamond$ is defined in this manner, then 
we have $(\iota,h)^\diamond(\xi_1)=(\iota,h)^\diamond(\eta_1)=0$ and
$(\iota,h)^\diamond(\xi_i)=\xi_i$, $(\iota,h)^\diamond(\eta_i)=\eta_i$ for $i=2,\cdots,l$. Then, we observe that $(\iota,h)^\diamond$ is compatible with the double complex $\K_{W/S,{\bf f}}^{p,{b,c}}$ decomposition, inducing morphism $(\iota,h)^{\diamond,double}: \K_{W/S,{\bf f}}^{p,{b,c}} \to \K_{W'/S,{\bf f}'}^{p,{b,c}}$ for all $p,b,c\in\Z$.
Then, the chain map in (2) obviously coincides with the one induced from $(\iota,h)^{\diamond,double}$.
\end{proof}

\medskip
Let us come back to the proof of {\it Lemma} \ref{f-free}.

\smallskip
\noindent
{\it Proof of Lemma} \ref{f-free} (1).  

Suppose that there are two complete intersection based charts $(j^1:U_1\to W_1,{\bf f}^1)$ and $(j^2:U_2\to W_2,{\bf f}^2)$ over the same base set $S$ and points $z_1\in U_1$ and $z_2\in U_2$ such that there is a local bi-holomorphic map $(U^1,z_1) \simeq (U^2,z_2)$ which commutes with the maps $\Phi_{U^1}$ and $\Phi_{U^2}$ in neighborhoods of $z_1$ and $z_2$. Then we show that there is a natural $\O_{U_1,z_1}$-$O_{U_2,z_2}$-isomorphism of the stalks:  
$$\mathcal{H}^{p,s}_{W^1/S,{\bf f^1},z_1} \simeq \mathcal{H}^{p,s}_{W^2/S,{\bf f^2},z_2}$$ which is equivariant with the De-Rham actions. By shrinking the relative charts $j^i$ ($i=1,2$) suitably, we may assume $U_1\simeq U_2$, and, further more, that $W_i$ is a Stein domain of $U_i\times \C^{l_i}$ such that i) the embedding $j^i$ is realized by the isomorphism $U_i\simeq U_i\times 0\subset W_i\subset U_i\times \C^{l_i}$ and ii) the composition of the embedding of $W_i$ in $U_i\times \C^{l_i}$ with the projection to the $j$-th component of $\C^{l_i}$ is equal to the $j$-th component, say $f^i_j$, of ${\bf f}^i$ ($i=1,2$) (however, the compositions of the embedding $W_i\to U_i\times \C^{l_i}$ with the projection to $U_i$ and with $\Phi_{U_i}$ may not necessarily coincide with the morphism $\Phi_{W_i}: W_i\to S$).  

The proof is achieved by introducing an auxiliarly  third based relative chart $(j,W)$. Namely, set $U:=U^1\simeq U^2$ and let $z\in U$ be the point corresponding to $z_i\in U_i$. Then, $W:=W_1\times_U W_2$ may naturally considered as a Stein domain in  $U\times \C^{l_1+l_2}$ such that $W_i=(U\times \C^{l_i})\cap W$ ($i=1,2$). Since $W$ is Stein and the maps $\Phi_{W_1}:W_1\to S$ and $\Phi_{W_2}:W_2\to S$ 
coincide with $\Phi_U$ on the intersection $W^1\cap W^2=U$, we can find a holomorphic map $\Phi_W:W\to S$ (up to some ambiguity) which coincides with $\Phi_{W_i}$ on each $W_i$ (e.g. $p_{W_1}^*\Phi_{W_1}+p_{W_2}^*\Phi_{W_2}-p_{U}^*\Phi_U$). We shall denote again by $f^1_j$ (resp.\ $f^2_j$) the $j$-th (resp.\ $l_1+j$-th) component of the coordinate of $\C^{l_1+l_2}$.Then, ${\bf f}:={\bf f}_1\cup{\bf f}_2$ forms a basis of the defining ideal $\I_U$ of $U\simeq U\times 0$ in $W$. Thus, we obtain a complete intersection based relative chart $(j:U\to W,{\bf f})$.

Let us show the existence of natural $\O_{U,z}$-isomorphisms: 
$$
{(***)} \qquad\qquad \qquad \qquad\qquad \qquad  \mathcal{H}^{p,s}_{W^i/S,{\bf f^i},z_i}\simeq \mathcal{H}^{p,s}_{W/S,{\bf f},z}  \qquad\qquad \qquad \qquad\qquad \qquad \qquad\qquad \qquad 
$$ 
commuting with De-Rham action for $i=1,2$. We show only the $i=1$ case (the other case follows similarly).  
For the end, we explicitly analyze the chain complex $(**)$ (see {\it Proof} of {\it Lemma} 4.7 in p.13) in a neighborhood of each point $z\in U$. 
Let $\underline{z}=(z^0,\cdots,z^n)$ be a local coordinate system of $U$ at $z$ so that $(\underline{z},{\bf f})$ form a coordinate system of $W$ at $z$. Let 
${\bf t}=(t^1,\cdots,t^{\dim S})$ be a local coordinate system of $S$ at the image of $z$, so that the morphism $\Phi_W:W\to S$ is expressed by the coordinates as ${\bf t}= \Phi_W(\underline{z},{\bf f}^1,{\bf f}^2)$ so that  $\Phi_{W_1}= \Phi_W(\underline{z},{\bf f}^1,0)$, 
$\Phi_{W_2}= \Phi_W(\underline{z},0,{\bf f}^2)$ and $\Phi_U=\Phi_W(\underline{z},0)$.

The fact that the restriction $\Phi_W |_{W_1}=\Phi_{W_1}$ is submersive over $S$ implies that already a $\dim_\C S$-minor of the part of Jacobi matrix of $\Phi_{W}(\underline{z},{\bf f}^1,{\bf f}^2)$ corresponding to the derivations by the coordinates $z^j$ ($j=0,\cdots,n$) and $f^1_1,\cdots,f_{l_1}^1$ is invertible (in a neighborhood of $z$). Then, in the quotient module $\Omega^1_{W/S}=\Omega^1_{W}/\sum_{i=1}^{\dim_\C S} \O_Wd\Phi_{W,i}$, the differentials $df_1^{2},\cdots, df_{l_2}^2$ of the remaining coordinates $f_1^2,\cdots,f_{l_2}^2$ form part of an $\O_W$-free basis in a neighborhood of $z$. Then, again shrinking the charts $W_i$ ($i=1,2)$ and $W$ suitably,  we can apply {\it Lemma} \ref{reduction} repeatedly, and we obtain the $\O_U$-isomorphism $(***)$. 

To show the independence of De Rham operator from a choince of basis ${\bf f}$, we cannot use the complex $(**)$ (there does not seem to exist a morphism $d_{DR}: (**)^p \to (**)^{p+1}$ which induces the De Rham operator: 
$\mathcal{H}^{p,s}_{W/S,{\bf f}} \to \mathcal{H}^{p+1,s}_{W/S,{\bf f}} $). However, {\it Lemma} \ref{reduction} (3) together with the naturallity of $\iota^\diamond$ ({\it Lemma} \ref{Koszul-cohomology} ii)) implies the compatibility of the De Rham operation with the isomorphism ($***$), and, hence, the independence from a choice of basis ${\bf f}$ of  the De Rham operator on $\H^{\bullet,s}_{W/S,{\bf f}}$.

\medskip
\noindent
{\it Proof of Lemma} \ref{f-free} (2).  
It is sufficient to show that the stalk of $\mathcal{H}^{p,s}_{W/S,{\bf f}} $ at a point, say $z$, of $U$, where $\Phi_U$ is submersive, vanishes for $s>0$. The assumption on the point $z$ means that the Jacobi matrix of $\Phi_U$ with respect to the derivations by $z^0,\cdots,z^n$ has a non-vanishing minor at the point $z\in U$.  So, in a neighborhood of $z$ in $W$, the corresponding minor of the Jacobi matrix of $\Phi_W$ does not vanish. This means that $df_1,\cdots,df_l$ form a part of $\O_W$-free basis of $\Omega_{W/S}^1$. Then applying {\it Lemma} \ref{reduction} inductively for a small neighborhood, we reduce to the relative chart of the form $j:U\to U$, and we conclude that $\mathcal{H}^{p,s}_{W/S,{\bf f}, z} $ is quasi-isomorphic to a single module $\Omega_{U/S,z}^p$ at $z$. That is, $\mathcal{H}^{p,0}_{W/S,{\bf f},z}\simeq  \Omega_{U/S,z}^p$ and $\mathcal{H}^{p,s}_{W/S,{\bf f},z}=0 $ for $s>0$.

This completes the proof of {\it Lemma} \ref{f-free}.
\end{proof}

\noindent
{\it Notation.}  As a consequence of {\it Lemma} \ref{f-free}, under the assumption that the relative chart $(j,{\bf f})$ is a complete intersection, the module  $\mathcal{H}^{p,s}_{W/S,{\bf f}}$, as an $\O_U$-module on $U$ with De Rham differential operator, depends only on the morphism $\Phi_U:U\to S$ but not on ${\bf f}$. Therefore, we shall denote the module also by $\mathcal{H}^{p,s}_{\Phi_U}$ (see {\it Lemma} \ref{Hpq}).

\bigskip
\begin{rem}
1.  The support of the modules $\mathcal{H}^{p,s}_{W/S,{\bf f}}$ for $s>0$ is contained in the critical set $C_{\Phi_U}$ (i.e.~locally, we have $\I_{C_{\Phi_U}}^m \mathcal{H}^{p,s}_{W/S,{\bf f}}=0$ for some positive integer $m$), does not imply that the module may not be an $\O_{C_{\Phi_U}}$-module.

2. In view of \cite{Saito11}, $\mathcal{H}^{p,s}_{W/S,{\bf f}}=0 $ for $s<\mathrm{depth} (\I_{\Phi_U})$. But we do not use this fact in the present paper.
\end{rem}

\bigskip
\noindent
{\bf E)   The complex  $(\mathcal{H}_\Phi^{\bullet,s}, d_{DR})$ on $Z$.}  

As an important consequence of  {\bf A)}-{\bf D)}, we introduce complexes 
$(\mathcal{H}_\Phi^{\bullet,s}, d_{DR})$ of  $\O_Z$-coherent sheaves  for $s\in \Z$.

\bigskip
\begin{lem}
\label{Hpq}
{\it 
Let $\Phi:Z\to S$ be a flat holomorphic map between complex manifolds and let $C_\Phi$ be its critical set loci as given in the Main Theorem. For $s\in \Z$, there  exists a  chain complex $(\mathcal{H}_\Phi^{\bullet,s},d_{DR})$ of $\O_Z$-coherent modules such that, 
for any based relative chart $(j:U\to W,{\bf f})$, there is a natural isomorphism: 
$$
(\mathcal{H}_\Phi^{\bullet,s}, d_{DR})|_{U} \simeq (\mathcal{H}_{W/S,{\bf f}}^{\bullet,s}, d_{DR}). 
$$
In particular, this implies

{\rm  i)}  For $s<0$, $\mathcal{H}_\Phi^{\bullet,s}=0$. 

{\rm ii)}  For $s=0$,
there is a natural isomorphism:   
 $$
 (\mathcal{H}_\Phi^{\bullet,0}, d_{DR})\quad \simeq \quad (\Omega_{Z/S}^\bullet,d_{DR}).
 $$

{\rm iii)}  For $s>0$ and $p\in\Z$, we have 
$$
\mathrm{Supp}(\mathcal{H}_\Phi^{p,s}) \quad \subset \quad C_{\Phi}.
$$
 }
\end{lem}
\begin{proof}  Let $(j:U\to W,{\bf f})$ be any based relative chart, which is a complete intersection.  Applying the construction of \eqref{Hps},  on the open subset $U$ of $Z$, we obtain a sequence for $s\in\Z$  of complexes $(\mathcal{H}_{W/S,{\bf f}}^{\bullet,s},d_{DR})$ of $\O_Z$-coherent modules equipped with the De Rham operator action. Let $(j':U'\to W',{\bf f}')$ be another complete intersection based relative chart, which introduces the complexes $(\mathcal{H}_{W'/S,{\bf f}'}^{\bullet,s},d_{DR})$ on the open set $U'$. Then {\it Lemma} \ref{f-free} together with $(***)$ says that, on the intersection $U\cap U'$, they patch each other naturally so that we obtain the complexes of sheaves on $U\cup U'$. Obviously, $Z$ is covered by charts which extends to complete intersection relative charts, there exists a global sheaf $ \mathcal{H}_\Phi^{p,s}$ on $Z$ together with the action of a De Rham operator as stated. The statement i) follows from the definition \eqref{Hps} and the fact $\mathcal{K}^{p,s}_{W/S,{\bf f}}\! =\! 0$ for $s\!<\!0$,  ii) follows from \eqref{Hp0}, and iii) follows from Lemma \ref{f-free} (2).
\end{proof}

\medskip
\begin{rem}
\label{Hpq-remark}
As we see, the chain complexes $(\mathcal{H}_\Phi^{\bullet,s}, d_{DR})$ themselves are independent of the choices of  relative charts. However, for its construction, we have used the relative charts. Can they be constructed without using the relative charts (or, without using Koszul-De Rham algebras)? (See the following Remark 4.12).
\end{rem}

\smallskip
\begin{rem}
It is also possible to consider a quotient algebra  $\overline{\mathcal{K}}_{W/S}$  of the Koszul-De Rham algebra $ \mathcal{K}_{W/S,{\bf f}}$ as follows.
Namely, suppose the defining ideal $\I_U$ of $U$ in $W$ has the following finite presentation.
\begin{equation}
\label{presentation}
\oplus\ \O_W^{l_1} \longrightarrow  \oplus\ \O_W^{l_0} \longrightarrow  \I_U  \longrightarrow 0.
\end{equation}
\noindent
Explicitly, let $f_1,\cdots\!, f_{l_0}\! \in\! \Gamma(W,\O_W)$ 
be a system generators of $\I_U$ (i.e.\ the image of the basis of $ \oplus\ \O_W^{l_0}$)  
and let $(g^1_j,\cdots, g^{l_0}_j)$ $\in \Gamma(W,\O_W^{l_0})$ ($j\! =\!1,\cdots\!,l_1$) 
be a generating system of relations $g^1_jf_1+\cdots+ g^{l_0}_jf_{l_0}\!=\!0$ (i.e.~ the image of the basis of  $\oplus\ \O_W^{l_1}$).
Then, we define
\begin{equation}
\label{quotient-algebra}
 \overline{\mathcal{K}}_{W/S}:= \Omega_{W/{S}}^\bullet \langle \xi_1,\cdots,\xi_{l_0}\rangle [\eta_1, \cdots,\eta_{l_0}] / \I
\end{equation}
where $\I$ is the both sided ideal  generated by the relations \eqref{even.odd} and \begin{equation}
\label{relations}
\begin{array}{cl}
g^1_j\xi_1+\cdots+ g^{l_0}_j\xi_{l_0} & (j=1,\cdots,l_1)  \\
g^1_j\eta_1+\cdots+ g^{l_0}_j \eta_{l_0} + dg^1_j\xi_1+\cdots+ dg^{l_0}_j\xi_{l_0} &(j=1,\cdots,l_1).
\end{array}
\end{equation}

Then, as the notation indicates, the algebra \eqref{quotient-algebra} does not depend on a choice of the presentation \eqref{presentation} of the ideal $\I_U$. Furthermore, it is not  hard to show that all the three structures Koszul differential $\partial_K$, De Rham differential $d_{DR}$ and the bi-degree structure $\mathcal{K}_{W/S,{\bf f}}^{p,s}$ are preserved on the quotient algebra $\overline{\mathcal{K}}_{W/S}^{p,s}$, and that a parallel  statement of the functoriality {\it lemma}s 4.1 and 4.2 hold, too. Then, for each fixed $p\in \Z$, we may also consider the cohomology of the $\partial_\K$. The following  question has quite likely a positive answer.

\noindent
{\bf Question.} Are the cohomology groups of $(\overline{\mathcal{K}}_{W/S}^{p,\star},\partial_\K)$ naturally isomorphic to those of $(\mathcal{K}_{W/S}^{p,\star},\partial_\K)$ (i.e.~to the groups $(\mathcal{H}_\Phi^{\bullet,s}, d_{DR})$ ($s\in \Z_{ge0}$))?
\end{rem}

\smallskip
\begin{rem} 
 In the present paper, we use the complexes $\mathcal{H}^{\bullet,s}_\Phi$ ($s\in \Z_{\ge0}$) only as a supporting actor for the proof of the coherence of the relative De Rham cohomology group of $\Phi$ (see Case 3. of \S5 {\bf D)}). 
But, for their definition, the condition that $\Phi|_{C_\Phi}$ is a proper map is un-necessary.
Therefore, we may expect a wider use of the complexes in future.
\end{rem}

\bigskip
\section{Step 4: Lifting of $\Check{\text{C}}$ech cohomology groups}

In this section, we give a final step of a proof of the Main Theorem: the coherence of the direct image $\R\Phi_*(\Omega^\bullet_{Z/S},d_{Z/S})$  in a neighborhood of any point $t\in S$ for a flat map $\Phi: Z\to S$ with a suitable boundary conditions.

We recall that
at Fact 2 of {\bf Step 2}, we showed that, for any point $t\in S$, there exists a Stein open neighborhood $S^*\subset S$ of $t$ and a finite system of relative charts $\mathfrak{U}:=\{ j_k: U_k\to D_k(1)\times S^*\}_{k=0}^{k^*}$ 
and 
a real number $0\!<\!r^*\!<\!1$ such that teh following holds:

\noindent
1) the intersection relative chart $j_K$ for $K\subset \{0,\cdots,k^*\}$ is complete intersection, 

\noindent
2)  for any Stein open subset $S'\subset S^*$ and $r^*\le \forall r\le 1$, consider the atlas 
$\U(r,S'):=\{ U_k(r,S'):=j_k^{-1}(D_k(r)\times S')\}_{k=0}^{k^*}$ \eqref{atlas1}
and the manifold $Z(r,S')$ $:=\cup_{k=0}^{k^*} U_k(r,S')$ covered by them. 
Then the direct image 
$\R\Phi(\Omega^\bullet_{Z(r,S')/S'})$ are isomorphic to each other for $r$ in $r^*\le r\le 1$.

 \smallskip
 \noindent
The plan of the proof is the following. 

\smallskip
{\bf A)} \ We  express the Hodge to De Rham spectral sequence over any Stein open subset $S'\subset S^*$ in terms of $\Check{\text{C}}$ech cohomology groups with coefficients in $\Omega^\bullet_{Z/S}$ with respect to the atlas $\U(r,S')$ \eqref{atlas1}.

{\bf B)} \ We ``lift" the $\Check{\text{C}}$ech complex to the lifted atlas $\mathfrak{U}(r,S')$  \eqref{atlas2} of relative charts. To be exact, in order to lift the coefficient to Koszul-De Rham algebra $\mathcal{K}_{W/S, \bf f}^{\bullet,\star}$, we need to enhance the atlas to a based lifted atlas $\tilde{\mathfrak{U}}(r,S')$. The existence of such enhancement shown in {\it Lemma} \ref{based-lifting} is a quite non-trivial step in the proof.

{\bf C)} \ We compare the $\Check{\text{C}}$ech complex  of $\Omega^\bullet_{Z/S}$  with that of $\mathcal{K}_{W/S, \bf f}^{\bullet,\star}$ by the morphism $\pi$ \eqref{0th-cohomology} and obtain a short exact sequence where the third term is described again by 
a $\Check{\text{C}}$ech complex with respect to the atlas $\U(r,S')$ and coefficient in the sheaf $\mathcal{H}^{\bullet,*}$  whose support is contained in the critical set $C_\Phi$. 

{\bf D)} \ In the long exact sequence of cohomology groups of the above three $\Check{\text{C}}$ech complexes, two terms (namely,  the first and the third) 
 are independent of the radius $r$. So the cohomology groups of the third, i.e.~of $\mathcal{K}_{W/S, \bf f}^{\bullet,\star}$, is also independent of $r$.

{\bf E)} \  We apply the Forster-Knorr Lemma (see \cite{Forster-Knorr} and also {\it Lemma} 5.2 of present paper) to the $\Check{\text{C}}$ech cohomology groups of $\mathcal{K}_{W/S, \bf f}^{\bullet,\star}$ and see that they give coherent direct image sheaves on a neighborhood $S_m$ of $t\in S$.  
On the other hand, the third term (the cohomology of $\mathcal{H}^{\bullet,*}$) is already coherent on $S_m$ since $C_\Phi$ is proper over $S$.  Thus, the remaining term in the long exact sequence of the cohomologies, that is, the  direct images of the relative De Rham complex  are also coherent on $S_m$.

\bigskip
\noindent
We start the proof now.

\smallskip
\noindent
{\bf A)  $\Check{\text{C}}$ech complex\ }

We consider the $\Check{\text{C}}$ech chain complex of the relative De Rham complex $\Omega^p_{Z(r,S')/S'}$ with respect to  a Stein covering $\U(r,S'):=\{U_k(r,S')\}_{k=0}^{k^*}$ \eqref{atlas1} of $Z(r,S')$ over any Stein open subset $S'\subset S^*$. As usual, the $q$th cochain module ($q\in\Z$) is given by 
\begin{equation}
\begin{array}{c}
\label{double}
\Check{\text{C}}^q(\U(r,S'),\Omega^p_{Z(r,S')/S'})
:= \underset{\substack{K\subset \{0,\cdots,k^*\}\\ \#K=q+1}}{\oplus}\Gamma(U_K(r,S'),\Omega^p_{Z'(r,S')/S'}) . 
\end{array}
\end{equation}
(where the summation index $K$ runs also over the cases when $U_K(r,S')=\emptyset$). The $\Check{\text{C}}$ech coboundary operator is the alternating sum 
\begin{equation}
\label{ceck-coboundary}
\Check{\delta}:=\!\!\sum_{K\subset K'} \pm (\rho^{K'}_K)^* : \  \Check{\text{C}}^q(\U(r,S'),\Omega^p_{Z(r,S')/S'}) \rightarrow \Check{\text{C}}^{q+1}(\U(r,S'),\Omega^p_{Z(r,S')/S'})
\end{equation}
of pull-back morphisms associated to the inclusion map $\rho_K^{K'}: U_{K'}\to U_K$, where $K$ and $K'\subset \{0,\cdots,k^*\}$ are indices satisfing $\#K=q+1$, $\#K'=q+2$ and $K\subset K'$.

\medskip
\noindent
{\bf B) Based lifting atlas\ }

Recall the lifting atlas $\mathfrak{U}(r,S'):=\{ j_k|_{U_k(r,S')} : U_k(r,S')\to D_k(r)\times S'\}_{k=0}^{k^*}$  \eqref{atlas2} of the atlas $\U(r,S')$ \eqref{atlas1}, where each $j_k$ is a pair $(\varphi_k,\Phi)$ of maps such that $\varphi_k$ is a local isomorphism of a neighborhood $B(\underline{z}_k,R(\underline{z}))$ of $\underline{z}_k\in Z$ to a domain in $\C^N$.
  We attach one more structure, i.e.~base (recall Definition 4), 
  to the atlas $\mathfrak{U}(r,S')$ as in the following definition. 

\bigskip
\begin{definition}  {\bf 6.}  A  {\it based lifting atlas} of  $\U(r,S')$ is a triplet 
$$
\widetilde{\mathfrak{U}}(r,S')\   :=\  (\mathfrak{U}(r,S'), {\bf f}_K, {\bf \Pi}_K^{K'}) \qquad 
$$ 
such that 

1) $\mathfrak{U}(r,S') =\{j_k \mid 0\le k \le k^*\}$ is the relative atlas already given in \eqref{atlas2},

2) ${\bf f}_K$ is a minimal generator system of the ideal $\I_{U_K(r,S*)}$ for $K\subset \{0,1,\cdots,k^*\}$. 
That is, $(j_K, {\bf f}_K)$ is a based relative chart in the sense of {\bf Definition 4}. 

3) ${\bf \Pi}_K^{K'}$  is a based morphism: $(j_{K'},{\bf f_{K'}}) \to (j_{K},{\bf f_{K}})$ (in the sense of {\it Lemma} \ref{functorial}) for $K, K'\subset \{0,1,\cdots,k^*\}$ with $K\subset K'$ 
such that  
$$
{\bf \Pi}^{K''}_K={\bf \Pi}^{K'}_K \circ {\bf \Pi}^{K''}_{K'}
$$ 
for any $K,K',K''\subset \{0,1,\cdots,k^*\}$ with $K\subset K'\subset K''$.
\end{definition}

\medskip
We remark that any based reltive chart in a based lifting atlas is automatically a complete intersection in the sense of {\bf Definition 5.}
The following existence of based lifting atlases is one crucial step towards the proof of the Main Theorem.
\bigskip

\begin{lem}
\label{based-lifting}
{\it For the atlas $\mathfrak{U}(r,S')$, there exists a based lifting $\widetilde{\mathfrak{U}}(r,S')$. }
\end{lem}
\begin{proof} We construct the based lifting explicitly in the following 1), 2) and 3).

\medskip
\noindent
1) Recall the notation of the proof of {\it Lemma} \ref{complete-intersection-atlas}. For any subset $K\subset \{0,1,\cdots,k^*\}$, we have 
$j_K : 
\underline{z}' \in U_K \mapsto ( (\varphi_{z_i}(\underline{z}'))_{i=0}^{k^*}, \Phi(\underline{z}'))\in \prod_{i=0}^{k^*}D_{z_i}(r_{z_i}) \times S_K$.  

\medskip
\noindent
2)  As was suggested already by 1) and 2) in the proof of {\it Lemma} \ref{complete-intersection-atlas} (page 6), we choose ${\bf f}_K$ as follows.

1.  If $U_K=\emptyset$, then we set  ${\bf f}_K=\{\ 1\ \}$.

2.  If $U_K\not=\emptyset$, then  ${\bf f}_K$ is the union of two parts ${\bf f}_{K,I}$ and ${\bf f}_{K,II}$ where
$$
{\bf f}_{K,I}= \left\{z^j\circ \varphi^{-1}_{k} - z^j\circ \varphi^{-1}_{k'} \right\}_{\substack{j=1, k,k'\in K}}^N \ \ \&\ \ {\bf f}_{K,II}=\left\{t_i-\Phi_i\circ \varphi_{k_0}^{-1} \right\}_{i=1}^{\dim_\C S},
$$
where  $k'=$the least element of $K$ which is larger than $k$, and  $k_0=\min\{ K\}$. 

\medskip
\noindent
3) Let  $K,K'\subset \{0,1,\cdots,k^*\}$ such that $K\subset K'$. We construct a morphism $\Pi_K^{K'}=(\pi_K^{K'},h_K^{K'}): (j_{K'},{\bf f}_{K'})\to (j_K,{\bf f}_K)$.  As a map $\pi_K^{K'}$ from the relative chart  $j_{K'}$ to $j_{K}$, we 
consider the pair consisting of natural projection: $\pi^{K'}_K\!\! :\! D_{K'}(r)\! \times\! S' \! \to \!  D_K(r)\! \times\!  S'$ and the natural (induced) inclusion: $U_{K'}\to U_K$. 

Let us choose and fix a morphism $h_K^{K'}$ between two basis ${\bf f}_K$  and ${\bf f}_{K'}$.  

In case $U_{K'}=\emptyset$, ${\bf f}_{K'}=\{1\}$ and we set  $h_K^1={\bf f}_K$. 

In case $U_{K'}\not=\emptyset$, then, according to the two groups of basis of 
${\bf f}_{K}$ and ${\bf f}_{K'}$
in the above 1), we decompose the matrix $h_K^{K'}$ into 4 blocks {\footnotesize $\left(\!\!\!\!\begin{array}{c}h_{K,I}^{K',I}, \ h_{K,I}^{K',II} \\ h_{K,II}^{K',I}, \ h_{K,II}^{K',II} \end{array}\!\!\!\! \right)$}, and fix the morphism blockwise in the following steps 1., 2. and 3.  

\smallskip
1.  There is a unique way to express any element of ${\bf f}_{K,I}$ as a sum of  elements of ${\bf f}_{K',I}$, then $h_{K,I}^{K',I} $ is its coefficients matrix. Thus we get: ${\bf f}_{K,I}=h_{K,I}^{K',I} {\bf f}_{K',I}$.

2.  We put $h_{K,I}^{K',II} = 0$.

3. We express ${\bf f}_{K,II} = h_{K,II}^{K',I} {\bf f}_{K',I} + h_{K,II}^{K',II} {\bf f}_{K',II}$, where $h_{K,II}^{K',II}$ is the identity matrix of size $\dim_\C S$. In order to fix the part $h_{K,II}^{K',I}$, we prepare some functions. 

For each $i$ with $1\le i\le \dim_\C S$, we express, locally in a Stein coordinate neighborhood,  $\Phi_i$ (the $i$th component of the map $\Phi$) as a function $\Phi_i(\underline z)$ of $N$ variables $\underline z=(z^1,\cdots, z^N)$. Consider a copy $\Phi_i(\underline{z}')$ of the function for a coordinate system $\underline{z}'=({z'}^1,\cdots, {z'}^N)$. Then, on the product domain of the coordinate  neighborhood, we can find functions $F_{ij}(\underline z,\underline z')$ ($j=1,\cdots,N$) such that
\begin{equation}
\label{difference}
\Phi_i({\underline z}') - \Phi_i({\underline z}) \quad =\quad \sum_{j=1}^N F_{ij}(\underline z ',\underline z) ({z'}^j-{z}^j),
\end{equation}
since the product domain is Stein where the ideal defining the diagonal is globally generated by ${z'}^j \!- \!{z}^j$ ($j\! =\! 1,\!\cdots\!,N$). Then, again taking a copy $\Phi_i(\underline{z}'')$ on the triple product domain and summing up two copies of above formula, we obtain a formula
\begin{equation}
\label{addition-formula}
\sum_{j=1}^N F_{ij}(\underline z'',\underline z) ({z''}^j-{z}^j) 
 = \sum_{j=1}^N F_{ij}(\underline z',\underline z) ({z'}^j-{z}^j) +  
\sum_{j=1}^N F_{ij}(\underline z'',\underline z) ({z''}^j-{z'}^j).
\end{equation}

\medskip
We return to the construction of the matrix $h_{K,II}^{K',I}$.
That is, we need to express the difference:
$
(t_i-\Phi_i\circ \varphi_{k_0}^{-1})-(t_i-\Phi_i\circ \varphi_{k'_0}^{-1}) =
\Phi_i\circ \varphi_{k'_0}^{-1} -\Phi_i\circ \varphi_{k_0}^{-1} 
$
as a linear combination of $z^j\circ \varphi_{k_0'}^{-1}-z^j\circ \varphi_{k_0}^{-1}$. The formula \eqref{difference} gives an answer:
$$
\Phi_i\circ \varphi_{k'_0}^{-1} -\Phi_i\circ \varphi_{k_0}^{-1}  
\quad =\quad 
 \sum_{j=1}^N F_{ij}(\varphi_{k_0'}^{-1},\varphi_{k_0}^{-1}) (z^j\circ\varphi_{k_0'}^{-1}-z^j\circ\varphi_{k_0}^{-1}), 
 $$
and we obtain the definition:  $h_{K,II}^{K',I}=\{ F_{ij}(\varphi_{k_0'}^{-1},\varphi_{k_0}^{-1}) \}_{i=1,\cdots,\dim_\C S,\ j=1,\cdots,N}$.

\medskip
Finally, we need to show that the above defined matrix  satisfies the functoriality
$h_{K}^{K''}=h_{K}^{K'}h_{K'}^{K''}$.  We can prove this again by decomposing the matrix into 4 blocks, where the cases of the blocks 
{\footnotesize $\left(\!\!\! \begin{array}{c} I\\ I\end{array} \!\!\! \right)$, $\left(\!\!\! \begin{array}{c} II\\ I\end{array} \!\!\!\right)$ } 
and 
{\footnotesize $\left(\!\!\! \begin{array}{c} II\\ II\end{array} \!\!\!\right)$}
 are trivial. The case of block 
 {\footnotesize $\left(\!\!\! \begin{array}{c} I\\ II\end{array} \!\!\!\right)$} 
 follows from the addition formula \eqref{addition-formula}.

This completes the proof of an existence of based lifting of the atlas $\U(r,S')$.
\end{proof}

\begin{rem} The above construction does not give a canonical lifting, but depends on the choices of the decomposition \eqref{difference} which is based on rather an abstract existence theorem (cf.~\cite{G-P-R}). We don't know the meaning of this freedom to the De Rham cohomology group we are studying. As  we see in sequel, for the proof of coherence, any choice of the lifting does work. See also Remark \ref{Hpq-remark}.
\end{rem}

\medskip 
From now on, we consider the base lifted atlas $\widetilde{\mathfrak{U}}(r,S')$ for all $r^*\le r\le 1$ and Stein open subset $S'\subset S^*$, depending on a choice of functions $F_{ij}$ in \eqref{difference}. 
Since we, later on, want to compare them for different $r$ and $S'$, we first fix the functions $F_{ij}$ and hence a based lifting $({\bf f}_K, {\bf \Pi}_K^{K'})$ on the largest atlas  $\mathfrak{U}(1,S^*)$,  then we consider the induced based lifting to any atlas 
 $\mathfrak{U}(r,S')$.

We lift the $\Check{\text{C}}$ech (co)chain complex \eqref{double} to the following triple chain complex. Namely,  for $p,q,s\in \Z_{\ge0}$,  we define the cochain module 
\begin{equation}
\label{triple}
\Check{\text{C}}^{q}(\widetilde{\mathfrak{U}}(r,S'), \K^{p,s}_\Phi)
:= \underset{\substack{K\subset \{0,\cdots,k^*\}\\ \#K=q+1}}{\oplus}\Gamma(D_K(r)\times S',\K_{D_K(r)\times S'/S',{\bf f}_K}^{p,s}).
\footnote{
In the notation of LHS, we replaced the subscript like $D_K(r)\! \times\! S/S$ indicating where the module is defined by $\Phi$, since we may regard $\K^{p,s}_\Phi$ to be a sheaf satisfying the functoriality ({\it Lemma} \ref{functorial}) defined on all relative charts, depending on the choice of a based lifting in {\it Lemma} \ref{based-lifting}.
}
\end{equation}
The actions of (co-)boundary operators $d_{DR}$ and 
$\partial_\K$ on the coefficient $\K^{\bullet,\star}_\Phi$ preserve the chart, so that they induce a double complex structure  
$(\Check{\text{C}}^{q}(\widetilde{\mathfrak{U}}(r,S'), \K^{\bullet,\star}_\Phi),d_{DR},d_{\K})$. 
We now lift the $\Check{\text{C}}$ech coboundary operator on \eqref{double} 
 to the lifted module \eqref{triple}.

\bigskip
For $p,q,s\in\Z$,  we introduce an {\it $\Gamma(S',\O_{S})$-homomorphism}
\begin{equation} 
\label{sign1}
\Check{\delta}  := \sum_{K\subset K'} \pm (\Pi^{K'}_K)^\diamond   \ : \ 
\Check{\text{C}}^{q}(\widetilde{\mathfrak{U}}(r,S'),\K^{p,s}_\Phi) \longrightarrow \Check{\text{C}}^{q+1}(\widetilde{\mathfrak{U}}(r,S'),\K^{p,s}_\Phi),
\end{equation} 
where 
$(\Pi^{K'}_K)^\diamond$ is the pull-back morphism (4.6) 
in {\it Lemma} \ref{functorial} associated with the morphism $\Pi_K^{K'}=(\pi_K^{K'},h_K^{K'})$  given in 3) of the proof of {\it Lemma} \ref{based-lifting},
and 
the sign and the running index $K$ and $K'$ are the same as those for the $\Check{\text{C}}$ech coboundary operator \eqref{ceck-coboundary}.  
We shall call this morphism the {\it lifted $\Check{\mathrm{C}}$ech coboundary operator}.  
 The lifted  $\Check{\mathrm{C}}$ech coboundary operator satisfies the relations:}
$$
\Check{\delta}^2=0, \quad 
\Check{\delta} d_{DR}+d_{DR}\Check{\delta}=0, \ \ \text{  and   }\ \  \Check{\delta} \partial_\K + \partial_\K\Check{\delta}=0.
$$
{\it Proof.}  To show that $\Check{\delta}^2=0$ is the same calculation as the standard $\Check{\text{C}}$ech coboundary case. Other relations follow from the fact that the pull-back homomorphism $(\pi_K^{K'})^\diamond$ 
commutes with $d_{DR}$ and $\partial_\K$ ({\it Lemma} \ref{functorial}).   \qquad $\Box$

\bigskip
\noindent
{\bf C) \ Comparison of the triple $\Check{\text{C}}$ech-complex of $\K_{\Phi}^{\bullet,\star}$ with the double $\Check{\text{C}}$ech-complex of $\Omega^\bullet_{\Phi}$} 

We compare the triple complex \eqref{triple} with the double complex \eqref{double}. More exactly,  for our restricted purpose (to calculate the second page of the Hodge to De Rham spectral sequence), we fix the index $p$ for the chain complex for De Rham differential operator. That is, we compare only the remaining double complex of the two coboundary operators $(\Check{\delta},\partial_\K)$ with the $\Check{\text{C}}$ech (co)chain complex of the coboundary operator $\Check{\delta}$. 
The comparison is achieved by the morphism $\pi$ (recall \S4 {\bf B)}).  
$$\cdots
\overset{\partial_\K}{\longrightarrow}
\Check{\text{C}}^{\ \!q}\!(\widetilde{\mathfrak{U}}(r,S'),\K^{p,1}_\Phi)  
\overset{\partial_\K}{\longrightarrow}
\Check{\text{C}}^{\ \!q}\!(\widetilde{\mathfrak{U}}(r,S'),\K^{p,0}_\Phi)  
\overset{\pi}{\longrightarrow} \Check{\text{C}}^{\ \!q}\!(\U,\Omega^p_{Z(r,S')/S'}) \to 0 .
$$
The commutativity of $\pi$ with  the lifted and un-lifted $\Check{\text{C}}$ech coboundary operator is, termwise, equivalent to the commutativity  $\rho^{U_K}_{U_{K'}}\circ \pi=\pi' \circ (\pi^{K'}_K)^\diamond$ \ \eqref{pi-natural}.

\bigskip

Let us consider the total complex of \eqref{triple} with respect to $\Check{\delta}$ and $\partial_\K$   by putting $\tilde *:=*-\star$ and $\tilde\partial:= \Check{\delta}+\partial_\K$ :
\begin{equation}
\label{total}
(Tot^{\tilde *}\Check{\text{C}}^{.}\!(\widetilde{\mathfrak{U}}(r,S'),\K^{p,.}_\Phi), \tilde\partial) \ \end{equation}
where
\begin{equation}
\label{total2}
Tot^{\tilde *}\Check{\text{C}}^{.}\!(\widetilde{\mathfrak{U}}(r,S'),\K^{p,.}_\Phi):=
\underset{\substack{ 
*-\star= \tilde *}  }{\oplus}
\Check{\text{C}}^{*}\!(\widetilde{\mathfrak{U}}(r,S'),\K^{p,\star}_\Phi).
\end{equation}

In view of \eqref{0th-cohomology}, for each fixed $p\in \Z$, the chain morphism 
\begin{equation}
\label{pi-final}
 (Tot^{\tilde *}\Check{\text{C}}^{.}\!(\widetilde{\mathfrak{U}}(r,S'),\K^{p,.}_\Phi), \tilde\partial) \ \overset{\pi}{\longrightarrow} \ 
(\Check{\text{C}}^*(\U,\Omega^p_{Z(r,S')/S'}), \Check{\delta}) 
\end{equation}
is an epimorphism in the category of cochain complexes. So, using the kernel of it, we obtain a short exact sequence:
\begin{equation}
\label{short}
\begin{array}{clr}
0 \to 
(Tot^{\tilde *}\Check{\text{C}}^{.}\!(\widetilde{\mathfrak{U}}(r,S'),\K^{p,.}_{\Phi,\ker(\pi)}), \tilde\partial) 
\! &\!\!\!  \overset{\iota}{\longrightarrow} 
(Tot^{\tilde *}\Check{\text{C}}^{.}\!(\widetilde{\mathfrak{U}}(r,S'),\K^{p,.}_\Phi), \tilde\partial) \\
&  \overset{\pi}
{\longrightarrow} \ 
(\Check{\text{C}}^*(\U,\Omega^p_{Z(r,S')/S'}), \Check{\delta}) \   \to \ 0, \!\!\!\!
\end{array}
\end{equation}
where the kernel (the first term) is again the total complex of a lifted $\check{\mathrm{C}}$ech chain complex of the atlas $\widetilde{\mathfrak{U}}(r,S')$ with coefficients in a complex $\K^{p,s}_{\Phi,\ker(\pi)}$ (for fixed $p$), which is the sub-complex of $\K^{p,s}_\Phi$ obtainded by replacing the first term $\K^{p,0}_\Phi$ by the term $\partial_\K({\K}^{p,1}_\Phi)= \ker(\pi: 
\K^{p,0}_\Phi\to \Omega^p_\Phi)$, and $\iota$ is the map induced from the natural inclusion $\K^{p,s}_{\Phi,\ker(\pi)} \subset \K^{p,s}_\Phi$. 

Due to the commutativity of $\pi$ with the De Rham differential operator \eqref{pi-deRham}, the chain maps  \eqref{pi-final} commute with De Rham operator action between the modules for the indices  $p$ and $p+1$. That is, by taking the direct sum over the index $p\in\Z$, we may regard $\pi$ as an epimorphism from the  double complex of $(\tilde\partial,d_{DR})$ to the double complex of $(\Check{\delta},d_{Z/S})$. Then, similarly, by taking the direct sum of the sequences \eqref{short} over the index $p\in \Z$, we obtain a short exact sequence of double complexes.

Before calculating  cohomology long exact sequence of the short exact sequence, we show in the following lemma some {\it finiteness} and {\it boundedness} of  the total complex \eqref{total} (considered as a double complex of the indices $\tilde *$ and $p$), which makes big contrast with the case of {\it Lemma} 4.2. Namely, in case of the total complex of $\partial_\K$ and $d_{DR}$, we did not get such finiteness and boundedness  (see Remark \ref{bound2}).  This finiteness, which holds for the total complex of $\partial_K$ and $\Check{\delta}$, is one of the most subtle but the key point where Koszul-De Rham algebra works mysteriously.

\bigskip
\noindent
\begin{lem}
\label{bound}
  {\it The complex \eqref{total} is finite and bounded in the following two senses.

\smallskip  
i)  The RHS  of  \eqref{total2} for fixed $p$ and $\tilde *$ is a finite direct sum of the form
$$
\overset{k^*-1}{ \underset{q=-1  }{\oplus} }\ 
\Check{\text{C}}^{q}\!(\widetilde{\mathfrak{U}}(r,S'),\K^{p,q-\tilde*}_\Phi)
 = 
 \underset{\substack{K\subset \{0,\cdots,k^*\}}}{\oplus}\Gamma(D_K(r)\times S',K_{D_K(r)\times S'/S',{\bf f}_K}^{p,\#K-\tilde{*}-1}).
$$

ii)  The set $\{ (p,\tilde *)\in \Z^2\mid Tot^{\tilde *}\Check{\text{C}}^{.}\!(\widetilde{\mathfrak{U}}(r,S'),\K_\Phi^{p,.})\not=0\}$ is contained in a strip }
\begin{equation}
\label{bound3}
- (k^*+1)(N-1) +n-1 \ \le \ \tilde * + p \ \le \ (k^*+1)(N+1) -1.
\end{equation}
\end{lem}
\begin{proof}  i)  The summation index $K\subset \{0,\cdots,k^*\}$ \eqref{triple} runs over a finite set so that $*=\#K-1$ is bounded. Then the condition that $*-\star= \tilde *$ is fixed means that the range of $\star$ is bounded. 

\smallskip
\noindent
ii)  Recall $Tot^{\tilde *}\Check{\text{C}}^{.}\!(\widetilde{\mathfrak{U}}(r,S'),\K^{p,.}_\Phi):=
\underset{*-\star= \tilde *}{\oplus} \underset{\#K=*+1}{\oplus}\Gamma(D_K(r)\times S',\K_{D_K(r)\times S'/S'}^{p,\star})$. If there is a non-vanishing term in RHS for some $*$,  $\star$ and $K$, then due to \eqref{bound1}, one has $-l_K\le p-\star \le \dim_\C D_K(r)$. Then, adding $*=\star+\tilde *$ in both hand sides, we have $-l_K+* \le \tilde *+p\le \dim_\C D_K(r)+*$.
Since $*=\#K-1$ and $l_K=\#K\cdot\dim_{\C}S+(\#K-1)n$, $\dim_{\C}D_K(r)=\#K\cdot N$ (recall \S3) and $\dim_{\C}Z=N=m=n+\dim_{\C}S$,
we get
$$
-\#K(N-1)+n-1 \ \le \ \tilde *+p \ \le \ \#K(N+1)-1
$$
Since the index $K$ runs over all subsets of $\{0,1,\dots,k^*\}$, we obtain the formula.
\end{proof}

According to i) and ii) of {\it Lemma} \ref{bound}, we have two important consequences: i) the cohomology groups is described by a finite chain complex where each chain module is a finite direct sum of the spaces of holomorphic functions on some relative charts (this descriptions is necessary to apply the Forster-Knorr Lemma), and ii) for each fixed $p$, the complex is bounded. This observation leads us to introduce the following truncation of the double complexes.
\begin{equation}
\label{truncation}
\begin{array} {ccll}
TR^{p,\tilde *} & := &
\begin {cases}
\begin{array}{cl}
Tot^{\tilde *}\Check{\text{C}}^{.}\!(\widetilde{\mathfrak{U}}(r,S'),\K^{p,.}_\Phi)  &\qquad \text{if \ $0\le p \le \dim_\C  Z$}\\
0 &\qquad  \text{otherwise}
\end{array}
\end{cases} \\
TR_{\ker(\pi)}^{p,\tilde *} & := &
\begin {cases}
\begin{array}{cl}
Tot^{\tilde *}\Check{\text{C}}^{.}\!(\widetilde{\mathfrak{U}}(r,S'),\K^{p,.}_{\Phi,\ker(\pi)})  & \text{if \ $0\le p \le \dim_\C  Z$}\\
0 & \text{otherwise}
\end{array}
\end{cases} \\
\end{array}
\end{equation}

For the truncated  double complexes, we have

i)  The complexes $TR^{p,\tilde *}$ and $TR_{\ker(\pi)}^{p,\tilde *}$ are bounded for the both indices $p$ and $\tilde *$, and also from above and below.

ii)  The following  is an exact sequence of bounded double complexes:
\begin{equation}
\label{goal}
0 \rightarrow  (TR_{\ker(\pi)}^{p,\tilde *}, d_{DR},\tilde\partial) \overset{\iota}{\longrightarrow} (TR^{p,\tilde *},d_{DR},\tilde\partial)
  \overset{\pi}
{\longrightarrow} \ 
(\Check{\text{C}}^*(\U,\Omega^p_{Z(r,S')/S'}), d_{Z/W},\Check{\delta}) \   \to \ 0, \
\end{equation}
which is the goal of our construction.   From now on, we start to analyze the sequence.

\bigskip
\noindent
{\bf D)  Long exact seqence of images on $S$.}

We consider now the long exact sequence of the cohomology group associated to the short exact  sequence obtained from \eqref{goal} by taking the total complex for each of the three double complexes. Recalling the construction of the atlases $\U(r,S')$ \eqref{atlas1} and $\mathfrak{U}(r,S')$ \eqref{atlas2}, we note that each term of the sequence depends on the choice of a set $S'$ and a real number $r$ with $r^*\le r \le 1$. By fixing $r$ and running $S'$ over all Stein open subset of $S^*$, we obtain a sheaf on $S^*$ (depending on $r$). 

In the following, we analyze the module (sheaf or its sections over $S'$) of the cohomology groups with the three coefficients cases separately.

\medskip
\noindent
{\bf Case 1.} $(\Check{\text{C}}^*(\U,\Omega^p_{Z(r,S')/S'}), d_{Z/W},\Check{\delta})$.

The module is exactly the module of relative De Rham hyper-cohomology group $\R \Phi_*\!(\Omega_{Z_{S^*}/{S^*}}^\bullet,\!d_{Z_{S^*}/{S^*}}) $ \eqref{r-independent} of the morphism $\Phi: Z(r) \to S^*$.  It is shown that the module is independent of the choice of $r$ with $r^*\le r\le 1$.

\smallskip
\noindent
{\bf Case 2.} $(TR^{p,\tilde *},d_{DR},\tilde\partial)$.

Due to the finiteness {\it Lemma} \ref{bound} i) and the boundedness of the double complex, the cohomology group is expressed as a cohomology group of a  finite complex, where each chain module is a finite direct sum of a module of the form 
$ \Gamma(D(r)\times S', \O_{D(r)\times S'}) $ for some polydisc $D(r)$ of radius $r$.

\noindent
{\it Proof.} Recall the direct sum decomposition \eqref{p-s-term}. Noting that $W/S$ is given by $D_K(r)\times S'/S'$ and we obtain $\Omega^p_{W/S}=\oplus_{i_1<\cdots <i_p}  \O_{D_K(r)\times S'} dz_{i_1}\wedge \cdots \wedge dz_{i_p}$ for a coordinate system $\underline{z}$ of the polydisc $D_K(r)$. $\Box$

\smallskip
\noindent
{\bf Case 3.} $ (TR_{\ker(\pi)}^{p,\tilde *}, d_{DR},\tilde\partial)$.

We approach the cohomology group of this case by a use of the spectral sequence of the double complex w.r.t. $d_{DR}$ and $\tilde\partial$. Let us  first calculate the cohomology group of the double complex with respect to the coboundary operator $\tilde\partial$ (in order to avoid a confusion, let us call the spectral sequence $E_I$). Thus, each entry of  the first page of the spectral sequence $E_I$ is again the total cohomology group of the total complex 
$(Tot^{\tilde *}\Check{\text{C}}^{.}\!(\widetilde{\mathfrak{U}}(r,S'),\K^{\bullet,.}_{\Phi,\ker(\pi)}), \tilde\partial=\Check{\delta}+\partial_K) $. Again we approach the group from the spectral sequence of the double complex w.r.t. $\Check{\delta}$ and $\partial_K$. 

Let us first consider the spectral sequence obtained by considering the cohomology group with respect to the coboundary operator $\partial_K$ first, for the reason below (let us call this spectral sequence $E_{II}$). 

  Recall {\it Lemma} \ref{Hpq} that it was shown that there exists a sequence of complexes $\mathcal{H}^{\bullet,s}_{\Phi}$ ($s\in \Z$) of coherent $\O_Z$--modules such that 1) the restriction of the $s$-th complex to $U_W$ induces a natural isomorphism to the $s$-th cohomology group of the Koszul-De Rham double complex $(\K^{\bullet,\star}_{W/S},d_{DR},\partial_K)$ with respect to the coboundary operator $\partial_K$, and 2) the support of the module for $s>0$ is contained in the critical set $C_\Phi$. 
 Thus, the $(q,s)$-entries of $E_{II}$ is given by direct images $\Check{\text{C}}^{q}\!(\widetilde{\mathfrak{U}}(r,S'),\mathcal{H}^{\bullet,s}_{\Phi}) $ of coherent sheaves 
  $\mathcal{H}^{\bullet,s}_{\Phi}$ (the fact that the pair $d_{DR}$ and $\Check{\delta}$ forms a double complex structure on $\oplus_{p,q} \Check{\text{C}}^{q}(\U(r,S'),\mathcal{H}_{\Phi}^{p,s}) $ is verified by a routine). 
  In view of the fact that $C_\Phi\subset Z'$ is proper over the base space $S$, this, in particular, implies that 1) the entry is 
  independent of $r$, and 2) the sheaf obtained by running $S'$ over all Stein open subset of $S^*$ is an $\O_{S^*}$-coherent module.  Then, these two properties should be inherited by the limit of the spectral sequence $E_{II}$ and the associated total cohomology group. 
  
  Coming back to the spectral sequence $E_I$, we see that all the entries of the first page of $E_I$ have the above properties 1) and 2). Thus the cohomology group of the total complex of the double complex  $ (TR_{\ker(\pi)}^{p,\tilde *}, d_{DR},\tilde\partial)$ should have the property.  Then in view of the long exact sequence, we started, two terms Case 1. and 3. of them (as a triangle) are independent of $r$. Thus, we conclude that the third term Case 2. satisfies:
  
{\it The total complex of the double complex $(TR^{p,\tilde *},d_{DR},\tilde\partial)$ is quasi-isomorphic to each other for $r$ and $r'$ with $r^*\le r, r'\le 1$. }

\bigskip
\noindent
{\bf  E)  Application of the Forster-Knorr  Lemma.} 

We are now able to apply the  following  key Lemma  due to Forster and Knorr \cite{Forster-Knorr} \cite{Knorr} (the formulation here of the result is taken from their unpublished note which is slightly modified from the published one, however can be deduced). 

 \medskip
\begin{lem} (Forster-Knorr) 
\label{Forster-Knorr}
{\it
Let $m$ be a given integer, ${S}$ a smooth complex manifold, $0$ a point in ${S}$.  
Suppose that $(C^{*}(r),d)$ is a complex of $\O_{{S}}$-modules bounded from the left such that 

i)  for any Stein open subset $S' \subset {S}$ and $q\in\Z$, we have an isomorphism
\[
C^q(r)(S') \simeq \prod_{finite} \Gamma(D(r)\times S', \O_{D(r)\times S'}) 
\]
together with the Fr\'echet topology. Here, $D(r)$ is a polycylinder of radius $r\in \R_{>0}$ whose dimension varies depending on each factor.

ii) $d: C^q(r)\to C^{q+1}(r)$ is an $\O_{{S}}$-homomorphism, which is continuous with respect to the Fr\'echet topology.

iii) There exist $r_1$ and $r_2$ such that, for any $r, \ r_1\ge r \ge r' \ge r_2>0$, the restriction $C^*(r)\!\to\! C^*(r')$ is a quasi-isomorphism.

\medskip
Then, there exists a small neighborhood ${S}_m$ of 0 in ${S}$ (depending on $m\in \Z$) such that , for $q\ge m$, 
$\mathrm{H}^q(C^*(r))\big|_{{S}_m}$ is an $\O_{{S}_m}$-coherent module.
}
\vspace{-0.1cm}
\end{lem}

We apply {\it Lemma} \ref{Forster-Knorr} to the total complex of complex $(TR^{p,\tilde *},d_{DR},\tilde\partial)$. 

Let us check that the complex satisfies the assumptions in the Forster-Knorr  Lemma
by putting $S=S^*$ (and run $S'$ over all Stein open subset of $S^*$ in order to make $\O_{S^*}$-module strcture), $0$ to be $t\in S^*$ and $r_1=1,\ r_2=r^*$. 

i) The condition i) is satisfied due to the description in {\bf D), Case} 
$(TR^{p,\tilde *},d_{DR},\tilde\partial)$.

ii) The condition ii) is verified as follows. The coboundary operator here is a mixture of $\partial_K, \Check{\delta}$ and $d_{DR}$, all of them are obviously $\O_S$-homomorphisms. That they are continuous w.r.t. the Fr\'echet topology can be seen as follows.

It is well known that the holomorphic function ring (Stein algebra) $\Gamma(D_K(r)\times {S}',\O_{D_K(1)\times {S}^*})$  carries naturally a Fr\'echet topology (\cite{Cartan}, \cite{G-P-R} p266) with respect to the compact open convergence. 
Then, the operators $\Check{\delta}$ and $\partial_\K$ are $\O_{D_K(1)\times {S}^*}$-homomorphisms and induce continuous morphisms on the modules. The  operator $d_{DR}$ is no longer an $\O_{D_K(1)\times {S}^*}$-homomorphism but is only an $\O_{{S}^*}$-homomorphism. Nevertheless, it is also well known that differentiation operators on a Stein algebra are also continuous w.r.t.~ the Fr\'echet topology.

iii) The quasi-isomorphisms between the complexes for $r$ and $r'$ with $r*\le r, r'\le 1$ was shown in the last step of {\bf D)}.

Finally, choosing $m=-1$, we obtain the coherence of the direct image sheaf of the total complex of $(TR^{p,\tilde *},d_{DR},\tilde\partial)$ in a neighborhood of $t\in S^*$.  
Then, we return to the long exact sequence studied in {\bf D)}. Two terms Case 2. and 3. of them (as a triangle) are $\O_S$-coherent near at $t\in S$. Therefore, the third term Case 3., the direct image of the double  complex, that is, {\it the hyper-cohomology groups $\R\Phi_*\!(\Omega_{Z/{S}}^\bullet,d_{Z/{S}})$ is also $\O_S$-coherent in a neighborhood of $t\in S$.}

This completes the proof  of the {\bf Main Theorem} given in Introduction.  \qed

\end{document}